\documentclass[10pt,a4paper]{article}
\usepackage[utf8]{inputenc}
\usepackage[english]{babel}
\usepackage{hyperref}
\usepackage{amsmath,amsthm,enumerate}
\usepackage{amssymb}
\usepackage{tikz}
\usepackage[hmargin=3cm,vmargin=3cm]{geometry}
\usepackage{mathrsfs}
\usepackage{subfigure}

\newtheorem{theorem}{Theorem}

\newtheorem{proposition}[theorem]{Proposition}
\newtheorem{lemma}[theorem]{Lemma}
\newtheorem{observation}[theorem]{Observation}
\newtheorem{corollary}[theorem]{Corollary}

\newtheorem{definition}[theorem]{Definition}
\newtheorem{question}[theorem]{Question}

\newcommand{\MD}{MD}
\newcommand{\LD}{LD}

\newcommand{\M}{CD}
\newcommand{\SC}{SC}

\newcommand{\optpb}[3]{
        \noindent\begin{minipage}{\textwidth}
                #1\\
                INSTANCE: #2\\ 
                TASK: #3
        \end{minipage}\vspace{\baselineskip}
}

\begin{document}

\title{Centroidal bases in graphs\thanks{The authors acknowledge the financial support from the Programme IdEx Bordeaux – CPU (ANR-10-IDEX-03-02).}}
\author{Florent Foucaud$^{a,b,c}$, Ralf Klasing$^d$ and Peter J. Slater$^{e,f}$\\ \\
{\small $a$: Universitat Politècnica de Catalunya, Carrer Jordi Girona 1-3, 08034 Barcelona, Spain.} \\
{\small $b$ Department of Mathematics, University of Johannesburg, Auckland Park 2006, South Africa.}\\
{\small $c$ LAMSADE - CNRS UMR 7243, PSL, Universit\'e Paris-Dauphine, F-75016 Paris, France.}\\
{\small $d$: CNRS - LaBRI UMR 5800 - Université de Bordeaux, F-33400 Talence, France.} \\
{\small $e$: Computer Science Department, University of Alabama in Huntsville, Huntsville, AL 35899, USA.} \\
{\small $f$: Mathematical Sciences Department, University of Alabama in Huntsville, Huntsville, AL 35899, USA.}\\
{\small (e-mails: {\tt florent.foucaud@gmail.com, ralf.klasing@labri.fr, pslater@cs.uah.edu/slaterp@uah.edu)}}
}

\maketitle

\begin{abstract}
We introduce the notion of a centroidal locating set of a graph $G$,
that is, a set $L$ of vertices such that all vertices in $G$ are
uniquely determined by their relative distances to the vertices of
$L$. A centroidal locating set of $G$ of minimum size is called a
centroidal basis, and its size is the centroidal dimension
$\M(G)$. This notion, which is related to previous concepts, gives a
new way of identifying the vertices of a graph. The centroidal
dimension of a graph $G$ is lower- and upper-bounded by the metric
dimension and twice the location-domination number of $G$,
respectively. The latter two parameters are standard and well-studied
notions in the field of graph identification.

We show that for any graph $G$ with $n$ vertices and maximum degree at
least~2, $(1+o(1))\frac{\ln n}{\ln\ln n}\leq \M(G) \leq n-1$. We
discuss the tightness of these bounds and in particular, we
characterize the set of graphs reaching the upper bound. We then show
that for graphs in which every pair of vertices is connected via a
bounded number of paths, $\M(G)=\Omega\left(\sqrt{|E(G)|}\right)$, the
bound being tight for paths and cycles. We finally investigate the
computational complexity of determining $\M(G)$ for an input graph
$G$, showing that the problem is hard and cannot even be approximated
efficiently up to a factor of $o(\log n)$. We also give an
$O\left(\sqrt{n\ln n}\right)$-approximation algorithm.

\end{abstract}

\section{Introduction}

A large body of work has evolved concerning the problem of identifying
an ``intruder'' vertex in a graph. As examples, one might seek to
identify a malfunctioning processor in a multiprocessor network, or
the location of an intruder such as a thief, saboteur or fire in a
graph-modeled facility. In this paper, we introduce the new model of
\emph{centroidal detection} as such a graph identification problem.

An early model considered the case where one could place detection
devices like sonar or LORAN stations at vertices in a graph; each
detection device could determine the distance to the intruder's vertex
location. As introduced independently in Slater~\cite{S75} and Harary
and Melter~\cite{HM76}, vertex set $L=\{w_1,\ldots,w_k\}\subseteq
V(G)$ is a \emph{locating set} (also called \emph{resolving set} in
the literature) if for each vertex $v\in V(G)$, the (ordered)
$k$-tuple $(d(v,w_1),\ldots,d(v,w_k))$ of distances between the
detector's locations and the intruder vertex $v$ uniquely determines
$v$. A minimum cardinality locating set is called a \emph{metric
  basis} (also called \emph{reference set} in the literature), and its
order is the \emph{metric dimension} of $G$, denoted by
$\MD(G)$. Other studies involving metric bases include
for example~\cite{BC,CEJO00,HN,KRR96,ST04}. Carson~\cite{Carson} and,
independently, Delmas, Gravier, Montassier and Parreau~\cite{DGMP11}
(for the latter authors, under the name of \emph{light $r$-codes})
considered the case in which each detection device at $w_i$ can only
detect an intruder at distance at most~$r$.

In another model, the presence of any edge $\{u,v\}\in E(G)$ indicates
that a detection device at $u$ is able to detect an intruder at $v$.
Let us denote by $N(u)$ and $N[u]$ the open and the closed
neighbourhood of vertex $u$, respectively.  A set $D\subseteq V(G)$ is
a \emph{dominating set} if $\bigcup_{u\in D}N[u]=V(G)$.  Clearly, in
the latter model, if every possible intruder location must be
detectable, the set of detector locations must form a dominating set.

The concepts of locating and dominating set were merged in
Slater~\cite{S87,S88}: when a detection device at vertex $u$ can
distinguish between there being an intruder at $u$ or at a vertex in
$N(u)$ (but which precise vertex in $N(u)$ cannot be determined), then
we have the concept of a \emph{locating-dominating set}. More
precisely, a set $D$ of vertices is locating-dominating if it is
dominating and every vertex in $V(G)\setminus D$ is dominated by a
distinct subset of $D$.  The minimum cardinality of a
locating-dominating set of graph $G$ is denoted $\LD(G)$. When one can
only decide if there is an intruder somewhere in $N[u]$, one is
interested in an \emph{identifying code}, as introduced by Karpovsky,
Chakrabarty and Levitin~\cite{KCL98}. Haynes, Henning and
Howard~\cite{HHH06} added the condition that the locating-dominating
set or identifying code not have any isolated vertices. When a
detection device at $u$ can determine that an intruder is in $N(u)$,
but will not report if the intruder is at $u$ itself, one is
interested in \emph{open-locating-dominating sets}, as introduced for
the hypercube by Honkala, Laihonen and Ranto~\cite{HLR02} and for all
graphs by Seo and Slater~\cite{SS10,SS11}. A bibliography of related
papers is maintained by Lobstein~\cite{biblio}.

In what follows, we will denote the path and the cycle on $n$ vertices
by $P_n$ and $C_n$, respectively.

In this paper, we introduce the study of \emph{centroidal bases}. In
this model, we assume that detection devices have unlimited range, as
for metric bases. However, exact distances to the intruder are not
known, but for detection devices $u,v$, if there is an intruder at
vertex $x$, then the presence of the intruder in the graph is
determined earlier by $u$ than by $v$ when $u$ is closer to $x$ than
$v$, that is, when $d(u,x)<d(v,x)$. When $d(u,x)=d(v,x)$ we assume that
$u$ and $v$ report simultaneously.

\begin{figure}[!htpb]
\centering
\scalebox{1.0}{\begin{tikzpicture}[join=bevel,inner sep=0.5mm,scale=1.0,line width=0.5pt]\path (0,0) node[draw,shape=circle,fill=black] (a) {};
\draw (a) node[above=0.1cm] {$x_1$};
\path (1,0) node[draw,shape=circle] (b) {};
\draw (b) node[above=0.1cm] {$x_2$};
\path (2,0) node[draw,shape=circle,fill=black] (c) {};
\draw (c) node[above=0.1cm] {$x_3$};
\path (3,0) node[draw,shape=circle] (d) {};
\draw (d) node[above=0.1cm] {$x_4$};
\path (4,0) node[draw,shape=circle] (e) {};
\draw (e) node[above=0.1cm] {$x_5$};
\path (5,0) node[draw,shape=circle,fill=black] (f) {};
\draw (f) node[above=0.1cm] {$x_6$};
\path (6,0) node[draw,shape=circle] (g) {};
\draw (g) node[above=0.1cm] {$x_7$};
\path (7,0) node[draw,shape=circle,fill=black] (h) {};
\draw (h) node[above=0.1cm] {$x_8$};
\draw (a) -- (b) -- (c) -- (d) -- (e) -- (f) -- (g) -- (h);
\end{tikzpicture}}
\caption{Centroidal basis $\{x_1,x_3,x_6,x_8\}$ for path $P_8$.}\label{fig:P8}
\end{figure}
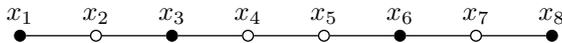

For example, consider $S=\{x_1,x_3,x_6,x_8\}\subseteq V(P_8)$ (see
Figure~\ref{fig:P8}). For vertex $x_4$, the order in which the
detectors of $S$ report is $(x_3,x_6,x_1,x_8)$ because
$d(x_4,x_3)=1<d(x_4,x_6)=2<d(x_4,x_1)=3<d(x_4,x_8)=4$. For vertex
$x_2$, we have $d(x_2,x_1)=d(x_2,x_3)$, hence the order of reporting
is $(\{x_1,x_3\},x_6,x_8)$. The smallest size of a set $S\subseteq
V(P_8)$ for which the order of reporting uniquely identifies each
vertex is, in fact, four.


\subsection{Medians and centroids}\label{sec:centroid}

In 1869, Jordan~\cite{J69} showed that each of the center and the (branch
weight) centroid of a tree either consists of one vertex, or of two
adjacent vertices. The \emph{eccentricity} of vertex $u$ is the
maximum distance from $u$ to another vertex in graph $G$,
$e(u)=\max\{d(u,v),v\in V(G)\}$, and the \emph{center} $\mathcal C(G)$
of $G$ is the set of vertices of minimum eccentricity. For a tree $T$,
the \emph{branch weight} of $u$, $bw(u)$, is the maximum number of
edges in a subtree with $u$ as an endpoint. The \emph{branch weight
  centroid} of $T$ is the set $\mathcal S_{bw}(T)$ of vertices with minimum
branch weight in $T$.

In 1964, Hakimi~\cite{H64} considered two facility location problems, one
involving the center. The second one involved the
\emph{distance} $d(u)=\sum_{v\in V(G)}d(u,v)$, measuring the total
response time at $u$ (this notion was called \emph{status} of $v$ by
Harary in 1959~\cite{H59}). The \emph{median} of $G$, $\mathcal
M(G)=\{u\in V(G),\forall v\in V(G), d(u)\leq d(v)\}$, is the set of
vertices of minimum distance in $G$. In 1968, Zelinka~\cite{Z68}
showed that for any tree $T$, $\mathcal M(T)=\mathcal S_{bw}(T)$,
which seemed to imply that the median would be a good generalization
of the branch weight centroid of a tree to an arbitrary graph. However
(with details in Slater~\cite{S80,S99} and Slater and Smart~\cite{SS99}),
note that components of $T-u$ of the same order, one being a path and
the other a star, contribute the same value of a branch weight, but
have much different distances. In trying to keep closer to the spirit
of what the branch weight centroid measures, the \emph{centroid}
$\mathcal S(G)$ of an arbitrary graph $G$ was defined in terms of
competitive facility location~\cite{S75b}. For facilities located at
vertices $u$ and $v$, a customer at vertex $x$ is interested in which
of the facilities is the closer. As defined in Slater~\cite{S75b}, the
set $V_{u,v}=\{x\in V(G), d(x,u)<d(x,v)\}$ is the set of vertex
customer locations strictly closer to $u$ than to $v$. Then,
$f(u,v)=|V_{u,v}|-|V_{v,u}|$ rates how well $u$ does as a facility
location, in comparison to $v$. Letting $f(u)=\min\{f(u,v), v\in
V(G)-u\}$, the \emph{centroid} of a graph $G$ is the set $\mathcal
S(G)=\{u\in V(G), \forall v\in V(G)-u, f(u)\geq f(v)\}$. When $G$ is a
tree, the centroid and branch weight centroid are easily seen to
coincide. Interestingly, there are graphs $G$ for which $f(u)<0$ for
all $u\in V(G)$.

Note that in our context of centroidal bases, detectors located at $u$
and $v$ enable us to determine if an intruder is in $V_{u,v}$, in
$V_{v,u}$, or in $V(G)-V_{u,v}-V_{v,u}=\{x\in V(G), d(x,u)=d(x,v)\}$. 

\subsection{Centroidal detection}\label{sec:def}

Let $B=\{w_1,\ldots,w_k\}\subseteq V(G)$ be a set of vertices of graph
$G$ with detection devices located at each $w_i$. As noted, we will
assume that each detection device has an unlimited range --- an
intruder entering at any vertex $x$ will, at some point, have its
presence noted at each $w_i$. Simply the presence will be noted, with
no information about the location of the intruder. In particular,
unlike in the setting of metric bases, $d(x,w_i)$ will not be
known. However, the time it takes before $w_i$ detects the intruder at
$x$ will be an increasing function of the distance $d(x,w_i)$. That
is, $w_i$ will indicate an intruder presence before $w_j$ whenever
$d(x,w_i)<d(x,w_j)$ (in our previous terminology, $x\in
V_{w_i,w_j}$). We will say that $x$ is \emph{located} first by $w_i$,
and then by $w_j$. Thus, each vertex $x$ has a rank ordering $r(x)$ of
the elements of a partition of $B$ (in fact, $r(x)$ is an \emph{ordered
  partition} of $B$, that is, an ordered set of disjoint subsets of
$B$ whose union is $B$). This ordering lists all the elements of $B$
in non-decreasing order by their distance from $x$, with ties
noted. Note that the number of ordered partitions of a set $B$ of $k$
elements is the \emph{$k$-th ordered Bell number}, denoted $b(k)$ (see
the book of Wilf \cite[Section 5.2, Example 1]{orderedBell}).

For $B=\{x_1,x_3,x_6,x_8\}$ in path $P_8$ (see Figure~\ref{fig:P8}),
$r(x_1)=(x_1,x_3,x_6,x_8)$, $r(x_2)=(\{x_1,x_3\},x_6,x_8)$,
$r(x_3)=(x_3,x_1,x_6,x_8)$, $r(x_4)=(x_3,x_6,x_1,x_8)$,
$r(x_5)=(x_6,x_3,x_8,x_1)$, $r(x_6)=(x_6,x_8,x_3,x_1)$,
$r(x_7)=(\{x_6,x_8\},x_3,x_1)$ and $r(x_8)=(x_8,x_6,x_3,x_1)$.

\begin{definition}
Vertex set $B\subseteq V(G)$ is called a \emph{centroidal locating set}
of graph $G$ if $r(x)\neq r(y)$ for every pair $x,y$ of distinct
vertices. A \emph{centroidal basis} of $G$ is a centroidal locating
set of minimum cardinality. The \emph{centroidal dimension} of $G$,
denoted $\M(G)$, is the cardinality of a centroidal basis.
\end{definition}

In our example, $B=\{x_1,x_3,x_6,x_8\}$ is the unique centroidal basis
of path $P_8$, and $\M(P_8)=4$.

Observe that every graph $G$ has a centroidal locating set,
for example, $V(G)$: each vertex $x$ is the only vertex having $x$ as the
first element of $r(x)$.

A useful reformulation of the definition of a
centroidal locating set is as follows:

\begin{observation}\label{obs:otherdef}
A set $B$ of vertices of a graph $V(G)$ is a centroidal locating set
if and only if for every pair $x,y$ of distinct vertices of $V(G)$,
there exist two vertices $b_1,b_2$ in $B$ such that either
$d(x,b_1)\leq d(x,b_2)$ but $d(y,b_1)>d(y,b_2)$, or $d(y,b_1)\leq
d(y,b_2)$ and $d(x,b_1)>d(x,b_2)$ (in other words, $y\in V_{b_2,b_1}$
but not $x$, or $x\in V_{b_2,b_1}$, but not $y$).
\end{observation}

\subsection{Structure of the paper}

We start in Section~\ref{sec:prelim} by stating some preliminary
observations and lemmas, and by giving bounds on parameter $\M$
involving the order, the diameter, and other parameters of graphs; in
particular, we show that $(1+o(1))\frac{\ln n}{\ln\ln n}\leq \M(G)
\leq n-1$ when $G$ has $n$ vertices and maximum degree at least~2.

In Section~\ref{sec:tightness}, we discuss the tightness of the two
aforementioned bounds by constructing graphs with small centroidal
dimension, and by fully characterizing the graphs having centroidal
dimension $n-1$.

In Section~\ref{sec:paths-cycles}, we give a lower bound
$\M(G)=\Omega\left(\sqrt{\frac{m}{k}}\right)$ when $G$ has $m$ edges and
every pair of vertices is connected by a small number, $k$, of paths.
We show that the bound is tight (up to a constant factor) for paths
and cycles.

Finally, in Section~\ref{sec:complex}, we discuss the computational
complexity of finding a centroidal basis; we show that for graphs with
$n$ vertices, it is \textsf{NP}-hard to compute an
$o(\ln n)$-approximate solution, and describe an
$O(\sqrt{n\ln n})$-approximation algorithm. We also remark that the
problem is fixed-parameter-tractable when parameterized by the
solution size.

\section{Preliminaries and bounds}\label{sec:prelim}

In this section, we give a series of preliminary lemmas and bounds for
parameter $\M$ that will prove useful later on, and also help the
reader become familiar with some of the aspects of the problem.

\subsection{Preliminary results}

We state a few lemmas that will prove very useful in the study
of centroidal locating sets.

\begin{lemma}
Let $G$ be a graph. The following statements are true:\\
(a) If $u$ is a vertex of degree~1, then any centroidal
locating set of $G$ contains $u$.\label{lemma:deg1}\\
(b) If $u$ is a vertex of degree~1 having a
neighbour $v$ of degree~2, then any centroidal locating set of $G$
contains either $v$ or a neighbour of $v$ other than $u$.\label{lemma:deg1-2}\\
(c) If $u,v$ are two vertices with $N(u)=N(v)$ or $N[u]=N[v]$, then any
centroidal locating set of $G$ contains at least one of $u$ and $v$.\label{lemma:twins}
\end{lemma}
\begin{proof}
(a): Otherwise, $u$ and its neighbour are not distinguished.

\noindent (b) and (c): Otherwise, $u$ and $v$ are not distinguished.
\end{proof}

\begin{lemma}\label{lemma:deg2}
Let $S$ be a set of vertices of a graph $G$ such that for each
$u\in S$, $|N(u)\setminus S|\geq 2$ and such that for each $u,v\in S$,
$N(u)\setminus S\neq N(v)\setminus S$. Then $V(G)\setminus S$ is a
centroidal locating set of $G$.
\end{lemma}
\begin{proof}Let $B=V(G)\setminus S$. Every vertex of $B$ is first located by
itself, while every vertex of $S$ is first located by a distinct set
of at least two vertices of $B$.
\end{proof}

Note that in particular, Lemma~\ref{lemma:deg2} shows that for any vertex $u$ of
degree at least~2 in a graph $G$, $V(G)\setminus\{u\}$ is a centroidal
locating set of $G$.

\subsection{Bounds}\label{subsec:bounds}

We now provide some lower and upper bounds for the value of parameter
$\M$.

\begin{theorem}\label{thm:LB}
Let $G$ be a graph on $n$ vertices with maximum degree at
least~2. Then $$(1+o(1))\frac{\ln n}{\ln\ln n}\leq \M(G) \leq n-1.$$
\end{theorem}
\begin{proof}
Lemma~\ref{lemma:deg2} immediately implies the upper bound. For the
lower bound, assume that $B$ is a centroidal basis of size $k=\M(G)$,
and $G$ has $n$ vertices. Then, to each vertex of $G$, one can assign
a distinct ordered partition of $B$. It is known that the number of
ordered partitions of a set of $k$ elements, the ordered Bell number
$b(k)$, is approximated by $$b(k)\sim\frac{k!}{2(\ln
  2)^{k+1}}+O(0.16^{k}k!),$$ see Wilf \cite[Section 5.2, Example
  1]{orderedBell}. It is clear that $n\leq b(k)$. Let us assume that
$B$ is a centroidal basis of $G$ of size $k$, with $n=b(k)$: for large
enough $n$, we have $n= k!(c^{k+1})$ for some constant $c$. Taking
the logarithm on both sides we get $\ln n=\ln (k!)+(k+1)\ln(c)$. By
using Stirling's approximation $\ln(k!)=k\ln(k)-k+O(\ln(k))$, we
obtain:
\begin{align}
\ln n=(1+o(1))k\ln(k).\label{eq:lnn}
\end{align}

Hence, $k=(1+o(1))\frac{\ln n}{\ln(k)}$; again taking the logarithm,
we get:
\begin{align}
\ln(k)=\ln\ln n-\ln\ln(k)+\ln(1+o(1))=(1+o(1))\ln\ln n.\label{eq:lnk}
\end{align}

Merging Equalities~\eqref{eq:lnn} and~\ref{eq:lnk}, we get:
$$k=(1+o(1))\frac{\ln n}{\ln\ln n}.$$

Since in $G$, there cannot be any smaller centroidal locating set than
$B$, the bound follows.
\end{proof}

The considerations of Theorem~\ref{thm:LB} can be strengthened if we
assume that the distances of a vertex to vertices in $B$ are bounded.
The following result was already known in the context of the metric
dimension, see e.g. Khuller, Raghavachari and Rosenfeld~\cite{KRR96}
or Chartrand, Eroh, Johnson and Oellermann~\cite{CEJO00}.\footnote{The
  result was stated in terms of the metric dimension but since any
  centroidal locating set is also a locating set the bound holds also
  for the centroidal dimension.}

\begin{proposition}\label{prop:diam}
Let $G$ be a graph with a centroidal locating set $B$ of size $k$ and
let $D$ be an integer. If for every vertex $u\in V(G)\setminus B$ and
for every vertex $b\in B$, $d(u,b)\leq D$, then $n\leq k+D^k$ and
hence $\M(G)\geq\log_D(n)-1$. In particular, this holds if $G$ has
diameter~$D$.
\end{proposition}
\begin{proof}
Let $u\in V(G)\setminus B$ be a vertex. Since every vertex of $B$ is
at distance at most $D$ from $u$, $r(u)$ contains at most $D$
sets. There are exactly $D^k$ different ordered partitions of $B$ into
at most $D$ sets (this number is equal to the number of words of
length $k$ over an alphabet of size $D$), hence there can be at most
$D^k$ vertices in $V(G)\setminus B$, and hence $k+D^k$ vertices in
$G$.
\end{proof}

The bound of Proposition~\ref{prop:diam} was improved by Hernando,
Mora, Pelayo, Seara and Wood~\cite{HMMPSW10}:

\begin{theorem}[\cite{HMMPSW10}]\label{thm:diam}
Let $G$ be a graph on $n$ vertices with diameter $D\geq 2$ and
$\M(G)=k\geq 1$. Then $n\leq
\left(\left\lfloor\frac{2D}{3}\right\rfloor+1\right)^k+k\sum_{i=1}^{\lceil
  D/3\rceil}(2i-1)^{k-1}$.
\end{theorem}

We improve the bound in Theorem~\ref{thm:diam} for $D\leq 3$:

\begin{theorem}\label{thm:diam3}
Let $G$ be a graph on $n$ vertices with diameter $D\in\{2,3\}$ and let
$\M(G)=k$. If $D=2$ and $k\geq 1$, $n\leq 2^{k}+k-1$. If $D=3$ and
$k\geq 5$, $n\leq 3^k-2^{k+1}+2$.
\end{theorem}
\begin{proof}
Let $B$ be a centroidal locating set of $G$.

\emph{$\bullet$ $D=2$ and $k\geq 1$.} By
Theorem~\ref{thm:diam} we have $n\leq 2^{k}+k$. However, observe that
for every vertex $v\notin B$, $r(v)=(N(v)\cap B, B\setminus
N(v))$. But there can only be $2^{k}$ distinct sets $N(v)\cap B$, and
moreover if $N(v)=B$ and $N(w)=\emptyset$ we have $r(v)=r(w)=(B)$. Hence
$|V(G)\setminus B|\leq 2^{k}-1$ and we are done.

\emph{$\bullet$ $D=3$ and $k\geq 5$.} Since the
diameter is~3, for every vertex $v$, $r(v)$ has at most three
components, unless $v\in B$, then it may have four. Moreover, if
$v\notin B$, then $r(v)$ either has one component (then $r(v)=(B)$),
or two (then $r(v)=(S,B\setminus S)$ with $1\leq |S|\leq |B|-1$), or
three (then $r(v)=(S,T,B\setminus (S\cup T))$ with $S\cap T=\emptyset$
and $1\leq |T|\leq |B\setminus S|$).

Therefore, we have
$$|V(G)\setminus B|\leq 1+\sum_{i=1}^{k-1}\left({k\choose i}(2^{k-i}-1)\right)=3^k-2^{k+1}+2,$$
that is, $1$ plus the number of ways of choosing a nonempty subset $S$ of $B$ and a nonempty subset of $B\setminus S$.

Now, if for every pair $b,b'$ in $B$ we have $d(b,b')\leq 2$, then no
vertex $v$ has four components in $r(v)$ and so the claimed bound
holds. Now, assume that there is a pair $b,b'$ in $B$ with
$d(b,b')=3$. This implies that for any vertex $v$, $r(v)$ cannot be of
the form $(\{b,b'\},T,B\setminus(\{b,b'\}\cup T))$ with $S$ a nonempty
  proper subset of $T\setminus\{b,b'\}$ since otherwise we must have
  $d(v,b)=d(v,b')=1$ and hence $d(b,b')\leq 2$, a contradiction. This
  gives us at least $k-2+{k-2\choose 2}$ forbidden triples of the
  form $(\{b,b'\},T,B\setminus(\{b,b'\}\cup T))$ since $T$ can be
    chosen to be one of $k-2$ possible singletons and ${k-2\choose 
      2}$ possible pairs. Since $k\geq 5$, $k-2+{k-2\choose 2}\geq k$
    and we are done.
\end{proof}

The bounds of Theorem~\ref{thm:diam3} will be proved tight in
Section~\ref{sec:tightness}.

We will now relate parameter $\M$ with parameters $\MD$ and $\LD$.

\begin{lemma}\label{lemm:LDset}
Let $G$ be a graph with a locating-dominating set $C$ such that $\ell$
vertices from $V(G)\setminus C$ have a unique neighbour in $C$. Then
$\M(G)\leq \LD(G)+\ell$.
\end{lemma}
\begin{proof}
We construct a centroidal locating set $C'$ from $C$ by adding at most
$\ell$ vertices to $C$. Note that in the setting of a centroidal
locating set and considering $C$ as a potential solution, each vertex
of $C$ is located first by itself, and then by its neighbourhood
within $C$, while each vertex $v$ of $V(G)\setminus C$ is first
located by $N(v)\cap C$. Hence, any two vertices both in $C$ or both
in $V(G)\setminus C$ are distinguished. However, if for some vertex
$v$ of $V(G)\setminus C$, $N(v)\cap C=\{c_v\}$, $v$ and $c_v$ might
not be distinguished. In that case, it is enough to add $v$ to $C$ to
solve this problem. This does not cause any other conflict since any
superset of a centroidal locating set is also a centroidal locating
set. Repeating the process $\ell$ times completes the proof (observing
that any other vertex of $C$ is distinguished from all other
vertices).
\end{proof}

Using Lemma~\ref{lemm:LDset}, we obtain the following theorem:

\begin{theorem}\label{thm:bounds-MD-LD}
For any graph $G$, $\MD(G)\leq\M(G)\leq 2\LD(G)$.
\end{theorem}
\begin{proof}
For the first inequality, note that any centroidal locating set $B$ is
a locating set. Indeed, if two vertices were at the same distance to
each vertex of $B$, then they would not be distinguished by their
relative distances by $B$, a contradiction.

The second inequality is proved by Lemma~\ref{lemm:LDset} by observing
that for any locating-dominating set $C$, $\ell\leq |C|$: for each
$c_v\in C$, if there were two vertices of $V(G)\setminus C$ having
only $c_v$ as a neighbour in $C$, they would not be distinguished by
$C$, a contradiction.
\end{proof}

For graphs of diameter~2, one gets the following improvement:

\begin{theorem}\label{thm:bound-diam2}
Let $G$ be a graph of diameter~2. Then $\LD(G)-1\leq\MD(G)\leq\M(G)\leq 2\LD(G)$.
\end{theorem}
\begin{proof}
The last two inequalities come from Theorem~\ref{thm:bounds-MD-LD}.
For the first inequality, we show that any locating set $L$ is almost
a locating-dominating set. Each vertex $v$ of $V(G)\setminus L$ has
distance~1 to all elements of its neighbourhood $N_v=N(v)\cap L$ in
$L$, and distance~2 to all vertices of $L\setminus N_v$. In other
words, vertices in $L$ only distinguish vertices they are adjacent to,
from non-adjacent ones. Since $L$ is a locating set, it follows that
each vertex in $V(G)\setminus L$ has a distinct neighbourhood within
$L$. Therefore, if $L$ is dominating, it is also
locating-dominating. Otherwise, there is at most one vertex that is
not dominated; adding it to set $L$, we get a locating-dominating set
of size $|L|+1$.
\end{proof}

\section{Tightness of the bounds}\label{sec:tightness}

In this section, we discuss the tightness of some of the bounds from
Subsection~\ref{subsec:bounds}.

\subsection{Graphs with small centroidal dimension}

For $k=1,2$ it is easy to construct graphs $G$ on $n$ vertices with
$\M(G)=k$ and $n=b(k)$: for $k=1$, $b(1)=1$ and $K_1$ is the only
answer; for $k=2$, $b(1)=3$ and both $P_3,K_3$ are answers.

For $k=3$, $b(3)=13$; the two graphs of Figure~\ref{fig:G13} have 13
vertices and centroidal dimension $3$ (the black vertices form a
centroidal basis). 

It holds that $b(4)=75$~\cite{sloane}; a more intricate construction
for this case is presented in Figure~\ref{fig:G75}.

 We do not know such optimal examples for $k\geq 5$ (recall that
 $b(k)$ grows very rapidly with $k$: $b(5)=541$,
 $b(6)=4683$~\cite{sloane}). Note that it is not possible to directly
 extend our example for $k=4$ to higher values by using the same idea;
 indeed, every two vertices from the centroidal basis $B$ are at
 distance at most~3 from each other. But in our construction, the
 vertices whose vector is a permutation of $B$ have a neighbour in the
 basis. Hence their vector $r$ can have length at most~4, but these
 $k!$ vertices need to have a vector of length $k$.

\begin{figure}[!htpb]
\centering
\subfigure[]{\scalebox{1.0}{\begin{tikzpicture}[join=bevel,inner sep=0.5mm,scale=0.8,line width=0.5pt]
\path (0:2cm) node[draw,shape=circle] (x0) {};
\path (1*30:2cm) node[draw,shape=circle,fill=black] (x1) {};
\path (2*30:2cm) node[draw,shape=circle] (x2) {};
\path (3*30:2cm) node[draw,shape=circle] (x3) {};
\path (4*30:2cm) node[draw,shape=circle] (x4) {};
\path (5*30:2cm) node[draw,shape=circle,fill=black] (x5) {};
\path (6*30:2cm) node[draw,shape=circle] (x6) {};
\path (7*30:2cm) node[draw,shape=circle] (x7) {};
\path (8*30:2cm) node[draw,shape=circle] (x8) {};
\path (9*30:2cm) node[draw,shape=circle,fill=black] (x9) {};
\path (10*30:2cm) node[draw,shape=circle] (x10) {};
\path (11*30:2cm) node[draw,shape=circle] (x11) {};

\path (0,0) node [draw,shape=circle] (c) {};
\draw[thick] (c) -- (x3)
             (c) -- (x7)
             (c) -- (x11)
             (x0) -- (x1) -- (x2) -- (x3)  -- (x4) -- (x5)  -- (x6) -- (x7)  -- (x8) -- (x9) -- (x10) -- (x11) -- (x0);
\end{tikzpicture}}}\qquad
\subfigure[]{\scalebox{1.0}{\begin{tikzpicture}[join=bevel,inner sep=0.5mm,scale=0.8,line width=0.5pt]
\path (0,0) node [draw,shape=circle] (c) {};
\path (-30:1.2cm) node[draw,shape=circle] (x0) {};
\path (30:1.2cm) node[draw,shape=circle,fill=black] (x1) {};
\path (90:1.2cm) node[draw,shape=circle] (x2) {};
\path (150:1.2cm) node[draw,shape=circle,fill=black] (x3) {};
\path (210:1.2cm) node[draw,shape=circle] (x4) {};
\path (270:1.2cm) node[draw,shape=circle,fill=black] (x5) {};
\path (x3)+(0,1.5) node[draw,shape=circle] (hg) {};
\path (x1)+(0,1.5) node[draw,shape=circle] (hd) {};
\path (x3)+(-1.3,-0.75) node[draw,shape=circle] (gh) {};
\path (x5)+(-1.3,-0.75) node[draw,shape=circle] (gb) {};
\path (x1)+(1.3,-0.75) node[draw,shape=circle] (dh) {};
\path (x5)+(1.3,-0.75) node[draw,shape=circle] (db) {};

\draw[thick] (c) -- (x3) -- (x2) -- (x1) -- (hd) -- (hg) --
             (x3) -- (gh) -- (gb) -- (x5) -- (db) -- (dh) --
             (x1) -- (c) -- (x5) -- (x0) -- (x1)
             (x5) -- (x4) -- (x3);

\end{tikzpicture}}}
\caption{Two graphs on $b(3)=13$ vertices with centroidal dimension $3$.}\label{fig:G13}
\end{figure}
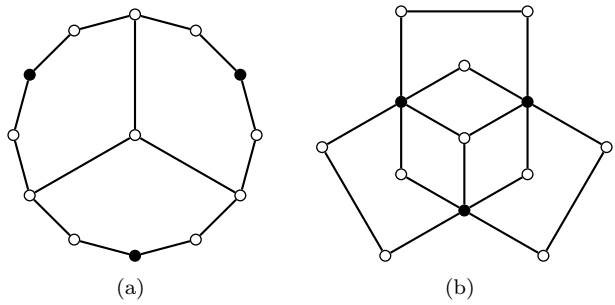

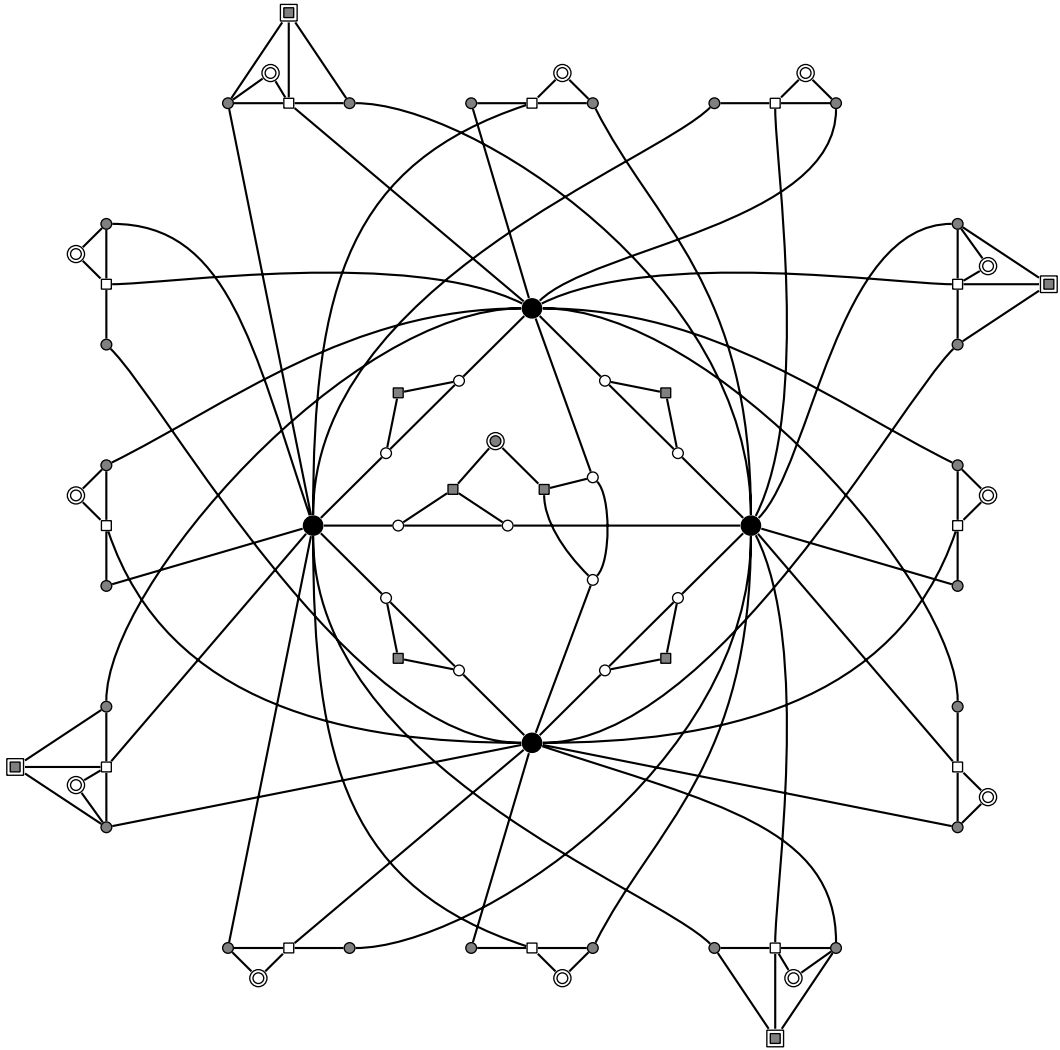
\begin{figure}[!htpb]
\centering
\scalebox{1.0}{\begin{tikzpicture}[join=bevel,inner sep=0.5mm,scale=0.8,line width=0.5pt]
\path (0,-3.6) node [draw,shape=circle,fill=black, scale=1.8] (c3) {};
\path (3.6,0) node [draw,shape=circle,fill=black, scale=1.8] (c2) {};
\path (0,3.6) node [draw,shape=circle,fill=black, scale=1.8] (c1) {};
\path (-3.6,0) node [draw,shape=circle,fill=black, scale=1.8] (c4) {};

\path (c4)+(1.2,1.2) node[draw,shape=circle] (a41) {};
\path (c4)+(2.4,2.4) node[draw,shape=circle] (a14) {};
\path (c2)+(-1.2,1.2) node[draw,shape=circle] (a21) {};
\path (c2)+(-2.4,2.4) node[draw,shape=circle] (a12) {};
\path (c4)+(1.2,-1.2) node[draw,shape=circle] (a43) {};
\path (c4)+(2.4,-2.4) node[draw,shape=circle] (a34) {};
\path (c2)+(-1.2,-1.2) node[draw,shape=circle] (a23) {};
\path (c2)+(-2.4,-2.4) node[draw,shape=circle] (a32) {};
\path (c3)+(1,2.7) node[draw,shape=circle] (a31) {};
\path (c3)+(1,4.4) node[draw,shape=circle] (a13) {};
\path (c4)+(1.4,0) node[draw,shape=circle] (a42) {};
\path (c4)+(3.2,0) node[draw,shape=circle] (a24) {};

\path (c4)+(1.4,2.2) node[draw,shape=rectangle, scale=1.3,fill=gray] (b41) {};
\path (c2)+(-1.4,2.2) node[draw,shape=rectangle, scale=1.3,fill=gray] (b21) {};
\path (c4)+(1.4,-2.2) node[draw,shape=rectangle, scale=1.3,fill=gray] (b43) {};
\path (c2)+(-1.4,-2.2) node[draw,shape=rectangle, scale=1.3,fill=gray] (b23) {};
\path (a13)+(-0.8,-0.2) node[draw,shape=rectangle, scale=1.3,fill=gray] (b31) {};
\path (a42)+(0.9,0.6) node[draw,shape=rectangle, scale=1.3,fill=gray] (b42) {};

\path (c4)+(3,1.4) node[draw,shape=circle, scale=1.6] (all) {};
\path (all) node[draw,shape=circle,fill=gray] {};

\path (-5,7) node[draw,shape=circle,fill=gray] (d412) {};
\path (d412)+(1,0) node[draw,shape=rectangle, scale=1.3] (e412) {};
\path (e412)+(1,0) node[draw,shape=circle,fill=gray] (d214) {};
\path (d412)+(0.7,0.5) node[draw,shape=circle, scale=1.6] (f412) {};
\path (f412) node[draw,shape=circle] {};
\path (d412)+(1,1.5) node[draw,shape=rectangle, scale=2.2] (g412) {};
\path (g412) node[draw,shape=rectangle, scale=1.3,fill=gray] {};

\path (-1,7) node[draw,shape=circle,fill=gray] (d142) {};
\path (d142)+(1,0) node[draw,shape=rectangle, scale=1.3] (e142) {};
\path (e142)+(1,0) node[draw,shape=circle,fill=gray] (d241) {};
\path (e142)+(0.5,0.5) node[draw,shape=circle, scale=1.6] (f142) {};
\path (f142) node[draw,shape=circle] {};

\path (3,7) node[draw,shape=circle,fill=gray] (d421) {};
\path (d421)+(1,0) node[draw,shape=rectangle, scale=1.3] (e421) {};
\path (e421)+(1,0) node[draw,shape=circle,fill=gray] (d124) {};
\path (e421)+(0.5,0.5) node[draw,shape=circle, scale=1.6] (f421) {};
\path (f421) node[draw,shape=circle] {};

\path (-5,-7) node[draw,shape=circle,fill=gray] (d432) {};
\path (d432)+(1,0) node[draw,shape=rectangle, scale=1.3] (e432) {};
\path (e432)+(1,0) node[draw,shape=circle,fill=gray] (d234) {};
\path (d432)+(0.5,-0.5) node[draw,shape=circle, scale=1.6] (f432) {};
\path (f432) node[draw,shape=circle] {};

\path (-1,-7) node[draw,shape=circle,fill=gray] (d342) {};
\path (d342)+(1,0) node[draw,shape=rectangle, scale=1.3] (e342) {};
\path (e342)+(1,0) node[draw,shape=circle,fill=gray] (d243) {};
\path (e342)+(0.5,-0.5) node[draw,shape=circle, scale=1.6] (f342) {};
\path (f342) node[draw,shape=circle] {};

\path (3,-7) node[draw,shape=circle,fill=gray] (d423) {};
\path (d423)+(1,0) node[draw,shape=rectangle, scale=1.3] (e423) {};
\path (e423)+(1,0) node[draw,shape=circle,fill=gray] (d324) {};
\path (e423)+(0.3,-0.5) node[draw,shape=circle, scale=1.6] (f423) {};
\path (f423) node[draw,shape=circle] {};
\path (d423)+(1,-1.5) node[draw,shape=rectangle, scale=2.2] (g432) {};
\path (g432) node[draw,shape=rectangle, scale=1.3,fill=gray] {};

\path (-7,-5) node[draw,shape=circle,fill=gray] (d341) {};
\path (d341)+(0,1) node[draw,shape=rectangle, scale=1.3] (e341) {};
\path (e341)+(0,1) node[draw,shape=circle,fill=gray] (d143) {};
\path (d341)+(-0.5,0.7) node[draw,shape=circle, scale=1.6] (f341) {};
\path (f341) node[draw,shape=circle] {};
\path (d341)+(-1.5,1) node[draw,shape=rectangle, scale=2.2] (g341) {};
\path (g341) node[draw,shape=rectangle, scale=1.3,fill=gray] {};

\path (-7,-1) node[draw,shape=circle,fill=gray] (d431) {};
\path (d431)+(0,1) node[draw,shape=rectangle, scale=1.3] (e431) {};
\path (e431)+(0,1) node[draw,shape=circle,fill=gray] (d134) {};
\path (e431)+(-0.5,0.5) node[draw,shape=circle, scale=1.6] (f431) {};
\path (f431) node[draw,shape=circle] {};

\path (-7,3) node[draw,shape=circle,fill=gray] (d314) {};
\path (d314)+(0,1) node[draw,shape=rectangle, scale=1.3] (e314) {};
\path (e314)+(0,1) node[draw,shape=circle,fill=gray] (d413) {};
\path (e314)+(-0.5,0.5) node[draw,shape=circle, scale=1.6] (f314) {};
\path (f314) node[draw,shape=circle] {};

\path (7,-5) node[draw,shape=circle,fill=gray] (d321) {};
\path (d321)+(0,1) node[draw,shape=rectangle, scale=1.3] (e321) {};
\path (e321)+(0,1) node[draw,shape=circle,fill=gray] (d123) {};
\path (d321)+(0.5,0.5) node[draw,shape=circle, scale=1.6] (f321) {};
\path (f321) node[draw,shape=circle] {};

\path (7,-1) node[draw,shape=circle,fill=gray] (d231) {};
\path (d231)+(0,1) node[draw,shape=rectangle, scale=1.3] (e231) {};
\path (e231)+(0,1) node[draw,shape=circle,fill=gray] (d132) {};
\path (e231)+(0.5,0.5) node[draw,shape=circle, scale=1.6] (f231) {};
\path (f231) node[draw,shape=circle] {};

\path (7,3) node[draw,shape=circle,fill=gray] (d312) {};
\path (d312)+(0,1) node[draw,shape=rectangle, scale=1.3] (e312) {};
\path (e312)+(0,1) node[draw,shape=circle,fill=gray] (d213) {};
\path (e312)+(0.5,0.3) node[draw,shape=circle, scale=1.6] (f312) {};
\path (f312) node[draw,shape=circle] {};
\path (d312)+(1.5,1) node[draw,shape=rectangle, scale=2.2] (g321) {};
\path (g321) node[draw,shape=rectangle, scale=1.3,fill=gray] {};

\draw[thick] (c4) -- (a41) -- (a14) -- (c1) --
             (a12) -- (a21) -- (c2) --
             (a23) -- (a32) -- (c3) --
             (a34) -- (a43) -- (c4) --
             (a42) -- (a24) -- (c2)
             (c1) -- (a13) .. controls +(0.3,-0.3) and +(0.3,0.3) .. (a31) -- (c3)
             (a41) -- (b41) -- (a14)
             (a21) -- (b21) -- (a12)
             (a43) -- (b43) -- (a34)
             (a23) -- (b23) -- (a32)
             (a13) -- (b31) .. controls +(0,-0.5) and +(-0.5,0.5) .. (a31)
             (a42) -- (b42) -- (a24)
             (b31) -- (all) -- (b42)
             (g412) -- (e412) -- (f412) -- (d412) -- (e412) -- (d214)
             (d412) -- (g412) -- (d214)
             (d412) -- (c4) (e412) -- (c1) (d214) .. controls +(2,0) and +(0,4) .. (c2)
             (e142) -- (f142) -- (d241) (d142) -- (e142) -- (d241)
             (d142) -- (c1) (e142) .. controls +(-3,-1) and +(0,4) .. (c4) (d241) .. controls +(1,-2) and +(0,4) .. (c2)
             (e421) -- (f421) -- (d124) (d421) -- (e421) -- (d124)
             (d421) .. controls +(-1,-1) and +(0,4) .. (c4) (e421) .. controls +(0,-1) and +(1,2) .. (c2) (d124) .. controls +(0,-2) and +(1,1) .. (c1)
                     (e432) -- (f432) -- (d432) -- (e432) -- (d234)
             (d432) -- (c4) (e432) -- (c3) (d234) .. controls +(2,0) and +(0,-4) .. (c2)
             (e342) -- (f342) -- (d243) (d342) -- (e342) -- (d243)
             (d342) -- (c3) (e342) .. controls +(-3,1) and +(0,-4) .. (c4) (d243) .. controls +(1,2) and +(0,-4) .. (c2)
             (g432) -- (e423) -- (f423) -- (d324) (d423) -- (e423) -- (d324)
             (d423) -- (g432) -- (d324)
             (d423) .. controls +(-1,1) and +(0,-4) .. (c4) (e423) .. controls +(0,1) and +(1,-2) .. (c2) (d324) .. controls +(0,2) and +(3,-1) .. (c3)  

             (g341) -- (e341) -- (f341) -- (d341) -- (e341) -- (d143)
             (d341) -- (g341) -- (d143)
             (d341) -- (c3) (e341) -- (c4) (d143) .. controls +(0,2) and +(-3,0) .. (c1)
             (e431) -- (f431) -- (d134) (d431) -- (e431) -- (d134)
             (d431) -- (c4) (e431) .. controls +(1,-3) and +(-3,0) .. (c3) (d134) .. controls +(2,1) and +(-3,0) .. (c1)
             (e314) -- (f314) -- (d413) (d314) -- (e314) -- (d413)
             (d314) .. controls +(1,-1) and +(-3,0) .. (c3) (e314) .. controls +(1,0) and +(-2,1) .. (c1) (d413) .. controls +(2,0) and +(-1,3) .. (c4)
 
             (e321) -- (f321) -- (d321) -- (e321) -- (d123)
             (d321) -- (c3) (e321) -- (c2) (d123) .. controls +(0,2) and +(3,0) .. (c1)
             (e231) -- (f231) -- (d132) (d231) -- (e231) -- (d132)
             (d231) -- (c2) (e231) .. controls +(-1,-3) and +(3,0) .. (c3) (d132) .. controls +(-2,1) and +(3,0) .. (c1)
             (g321) -- (e312) -- (f312) -- (d213) (d312) -- (e312) -- (d213)
             (d312) -- (g321) -- (d213)
             (d312) .. controls +(-1,-1) and +(3,0) .. (c3) (e312) .. controls +(-1,0) and +(2,1) .. (c1) (d213) .. controls +(-2,0) and +(1,1) .. (c2) 
;

\end{tikzpicture}}
\caption{A graph on $b(4)=75$ vertices with centroidal dimension $4$. In the central part of the figure, the 4 black vertices form the centroidal basis and have their vector $r$ of the form $(a,\{b,c,d\})$; they are at distance~3 from each other. The 12 white circle-shaped vertices have their vector of the form $(a,b,\{c,d\})$. The 6 gray square-shaped vertices have their vector of the form $(\{a,b\},\{c,d\})$. The unique gray double-circled vertex has its vector of the form $(\{a,b,c,d\})$. 
In the outer parts of the figure, the 24 gray circled-shaped vertices have their vector of the form $(a,b,c,d)$. The 12 white square-shaped vertices have their vector of the form $(a,\{b,c\},d)$. The 12 white double-circled vertices have their vector of the form $(\{a,b\},c,d)$. Finally, the 4 gray double-squared vertices have their vector of the form $(\{a,b,c\},d)$.}\label{fig:G75}
\end{figure}

\medskip
\noindent\textbf{Bounded diameter.} We now discuss the tightness of
the bounds for diameter~2 and~3 of Theorem~\ref{thm:diam3}.

\begin{proposition}\label{prop:diam2-construct}
For any $k\geq 4$, there is a graph $G$ of diameter~2 with $n=2^k+k-1$
vertices and $\M(G)=k$.
\end{proposition}
\begin{proof}
We construct $G$ in the following way. We let $V(G)=B\cup S$, where $B,S$
are disjoint sets, $B$ has $k$ vertices, and $S$ has $n-k=2^k-1$
vertices. We make sure that in the subgraph induced by $B$, every
vertex has a neighbor and a non-neighbor, which is possible since
$k\geq 4$. The set $S$ induces a clique and the neighborhood of every
vertex of $S$ within $B$ is distinct and nonempty.

Since $G$ has a vertex of degree $n-1$ (in $S$) the diameter is~2. For
every vertex $v$ in $B$, $r(v)=(\{v\},N(v)\cap B,B\setminus N[v])$,
and by our assumption on the structure of $B$, these three sets are
nonempty. For every vertex $v$ in $S$, $r(v)=(N(v)\cap B, B\setminus
N(v))$. By construction all these vectors are distinct.
\end{proof}

In fact, in the construction of
Proposition~\ref{prop:diam2-construct}, $B$ is a locating-dominating
set; similar constructions are well-known in this context, see for
example Slater~\cite{S88}. For diameter~3, we give a more complicated
construction:

\begin{theorem}\label{thm:diam3-construct}
For any $k\geq 4$, there is a graph $G$ of diameter~3 with
$n=3^k-2^{k+1}+2$ vertices and $\M(G)=k$.
\end{theorem}
\begin{proof}
We construct $G$ as follows (see Figure~\ref{fig:diam3} for an
illustration).
\begin{itemize}
\item Let $B$ be an independent set of size~$k$, which will be the
  centroidal locating set of $G$.
\item Let $X$ be a clique containing, for every subset $S$ of $B$ with
  $2\leq |S|\leq k-2$, a vertex $x(S)$ that is adjacent to all
  vertices in $S$. Set $X$ has size $\sum_{i=2}^{k-2}{k\choose i}$.
\item Let $Y$ be an independent set containing, for every subset $S$
  of $B$ with $1\leq |S|\leq k-2$ and for every proper nonempty subset
  $T$ of $B\setminus S$, a vertex $y(S,T)$ (note that $1\leq |T|\leq
  k-2$). Vertex $y(S,T)$ is adjacent to all vertices of $S$. Moreover,
  if $|T|\geq 2$, $y(S,T)$ is adjacent to $x(T)$. If $T=\{t\}$ has
  size~1, let $T'$ be an arbitrary size~2-subset of $B$ formed by $t$
  and an arbitrary vertex of $S$, and let $y(S,T)$ be adjacent to
  $x(T')$. Note that set $Y$ has size $\sum_{i=1}^{k-2}\left({k\choose
    i}(2^{k-i}-2)\right)$.
\item Let $Z$ be a clique of size $k+1$ containing, for each subset $S$ of $B$ with
  $k-1\leq |S|\leq k$, a vertex $z(S)$ that is adjacent to the
  vertices in $S$.
\end{itemize}

\begin{figure}[htpb!]
\centering
\includegraphics[width=0.55\textwidth]{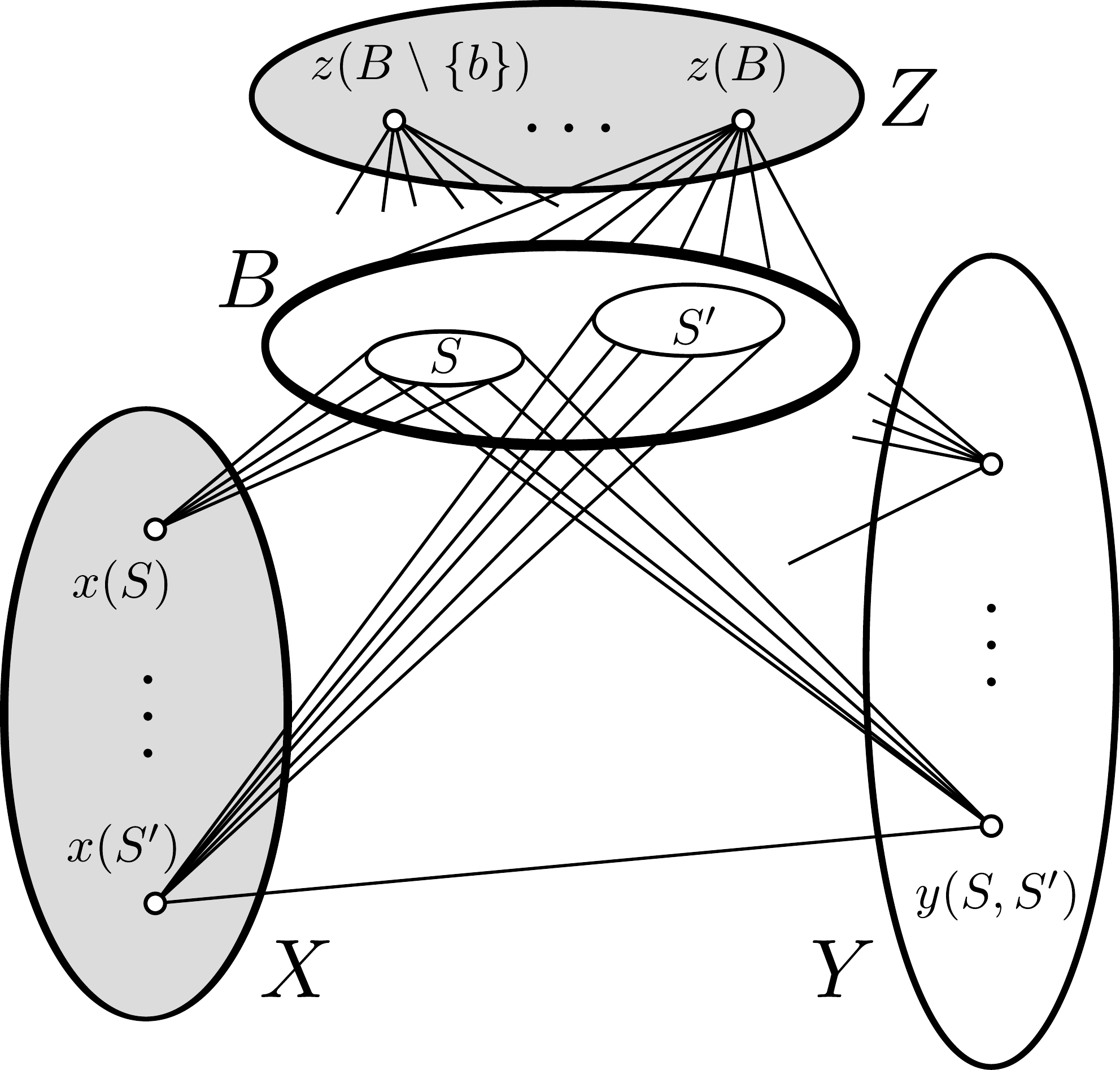}
\caption{The construction of Theorem~\ref{thm:diam3-construct}. Gray sets are cliques.}
\label{fig:diam3}
\end{figure}

The order of $G$ is

\begin{align*}
|B|+|X|+|Y|+|Z| &= k+\sum_{i=2}^{k-2}{k\choose i}+\sum_{i=1}^{k-2}\left({k\choose
    i}(2^{k-i}-2)\right)+k+1\\ &= 1+\sum_{i=1}^{k}\left({k\choose i}(2^{k-i}-1)\right) \\& =3^k-2^{k+1}+2.
\end{align*}

Furthermore, the diameter of $G$ is exactly~3. It is at least~3:
consider some vertex $y=y(S,T)$. We have $N(y)=S$, and the vertices in
$T$ are at distance~2 of $y$ (via vertex $x(T)$). However, all
vertices of $B\setminus (S\cup T)$ are at distance~3 of $y$, and since
$T$ is a proper subset of $B\setminus S$, $B\setminus (S\cup T)$ is
nonempty. On the other hand, the diameter is at most~3. If $v$ is any
vertex of $G$, $v$ is within distance~1 of some vertex $b$ in $B$, and
since $X$ and $Z$ are cliques and $b$ has a neighbor in both $X$ and
$Z$, every vertex in $X\cup Z$ is within distance~2 of $b$ and within
distance~3 of $v$. Moreover, every two vertices in $B$ are at distance
at most~2 away since they all share a neighbor, $z(B)$. Hence, any
vertex $y(S,T)\in Y$, since it has a neighbor in $B$, is at distance
at most~3 from any vertex in $B$. Finally, any vertex in $Y$ has a
neighbor in $X$, which is a clique; hence any two vertices in $Y$ have
distance at most~3 from each other.

It remains to check that $B$ is a centroidal locating set. For every
$b\in B$, we have $r(b)=(\{b\},B\setminus\{b\})$. For every vertex
$x=x(S)$ in $X$, we have $r(x)=(S,B\setminus S)$ and $2\leq |S|\leq
k-2$. For every vertex $z=z(S)$ in $Z$, $r(z)=\{S,B\setminus S\}$ if
$|S|=k-1$ and $r((z(B))=(B)$. We have now realised all possible
vectors with at most two components. Finally, for any vertex
$y=y(S,T)$ in $Y$, $r(y)=(S,T,B\setminus (S\cup T))$ and by the
definition of $Y$, none of these three components is empty. This
completes the proof.
\end{proof}

We leave the question of the tightness of Theorems~\ref{thm:LB} and~\ref{thm:diam} open.

\begin{question}\label{quest:LB}
Is the bound $\M(G)\geq(1+o(1))\frac{\ln n}{\ln\ln n}$
asymptotically tight, that is, can we find an infinite family of graphs
$G$ with $\M(G)=O\left(\frac{\ln n}{\ln\ln n}\right)$?
\end{question}

Observe that, by Proposition~\ref{prop:diam},
$\M(G)\geq\log_D(n)-1$ when $G$ has diameter $D$ and $n$
vertices. Hence, in order to construct a graph with a centroidal
locating set of size $O\left(\frac{\ln n}{f(n)}\right)$ for some
$f(n)=O(\ln\ln n)$, $G$ should have diameter
$\Omega\left(e^{f(n)}\right)$; in particular, for $f(n)=\ln\ln n$,
the diameter should be $\Omega\left(\ln n\right)$.

It was proved by Seb\H{o} and Tannier~\cite{ST04} that for the
hypercube $Q_k$ with $n=2^k$ vertices,
$\MD(Q_k)=\frac{2\log_2 n}{\log_2\log_2 n}(1+o(1))$. In this regard, it would be
interesting to determine whether the family of hypercubes is a
positive answer to Question~\ref{quest:LB}; determining $\M(Q_k)$
would be of independent interest.

\begin{question}\label{quest:LB-diam}
What is the maximum order of a graph with centroidal dimension~$k$ and
diameter~$D\geq 4$?
\end{question}

\subsection{Graphs with large centroidal dimension}

A direct consequence of Theorem~\ref{thm:LB} is that a graph has
centroidal dimension equal to its order if and only if it has maximum
degree at most~1. We now fully characterize the set of graph with
centroidal dimension of value the order minus one.

For some $n\geq 1$, $K_n$ denotes the complete graph on $n$
vertices. For some $a,b\geq 1$, $K_{a,b}$ denotes the complete
bipartite graph with parts of sizes $a$ and $b$. Let $n\geq 4$. We
denote by $S_n$, the graph obtained by joining $K_2$ to an independent
set of $n-2$ vertices. We call $T_n$ the tree obtained from $P_3$ by
attaching $n-3$ degree~1-vertices to one of the ends of
$P_3$. Finally, we call $U_n$ the graph obtained from $K_3$ by
attaching $n-3$ degree~1-vertices to one of the vertices of $K_3$.

In particular, $K_{1,n-1}$ is a star, $K_{2,2}$ is the cycle $C_4$,
$S_4$ is the diamond graph, and $T_4$ is the path $P_4$. See
Figure~\ref{fig:extremal} for illustrations of these graphs (black
vertices belong to a centroidal basis).

\begin{figure}[!htpb]
\centering
\subfigure[$K_{n}$]{
\scalebox{1.0}{\begin{tikzpicture}[join=bevel,inner sep=0.5mm,scale=0.8,line width=0.5pt]
\path (0:2cm) node[draw,shape=circle] (x0) {};
\path (2*60:2cm) node[draw,shape=circle,fill=black] (x1) {};
\path (3*60:2cm) node[draw,shape=circle,fill=black] (x2) {};
\path (4*60:2cm) node[draw,shape=circle,fill=black] (x3) {};
\path (5*60:2cm) node[draw,shape=circle,fill=black] (x4) {};

\draw[thick] (x0) -- (x1) -- (x2) -- (x3) -- (x4) -- (x0) -- (x2) -- (x4) -- (x1) -- (x3) -- (x0)
      (x3)+(2,2.5) -- (x3)
      (x3)+(1.5,3) -- (x3)
      (x3)+(1.1,3.1) -- (x3)
      (x4)+(0.1,2.5) -- (x4)
      (x4)+(-0.7,3) -- (x4)
      (x4)+(0.6,2.9) -- (x4)
      (x2)+(2.5,0.4) -- (x2)
      (x2)+(2.3,1.2) -- (x2)
      (x2)+(2.8,0.7) -- (x2);
\draw (1,1.4) node[rotate=-30] {$\ldots$};
\end{tikzpicture}}}\qquad
\subfigure[$K_{1,n-1}$]{
\scalebox{1.0}{\begin{tikzpicture}[join=bevel,inner sep=0.5mm,scale=0.8,line width=0.5pt]
\path (0:2cm) node[draw,shape=circle,fill=black] (P0) {};
\path (2*60:2cm) node[draw,shape=circle,fill=black] (P1) {};
\path (3*60:2cm) node[draw,shape=circle,fill=black] (P2) {};
\path (4*60:2cm) node[draw,shape=circle,fill=black] (P3) {};
\path (5*60:2cm) node[draw,shape=circle,fill=black] (P4) {};
\path (0,0) node [draw,shape=circle] (C) {};
\draw[thick] (P0) -- (C) 
      (P1) -- (C) 
      (P2) -- (C) 
      (P3) -- (C) 
      (P4) -- (C)
      (C)+(1,0.8) -- (C)
      (C)+(0.7,1.1) -- (C)
      (C)+(0.35,1.2) -- (C);
\draw (C)+(1,1.4) node[rotate=-30] {$\ldots$};

\end{tikzpicture}}}\qquad
\subfigure[$K_{2,n-2}$]{
\scalebox{1.0}{\begin{tikzpicture}[join=bevel,inner sep=0.5mm,scale=0.8,line width=0.5pt]
\path (0,0) node[draw,shape=circle,fill=black] (c0) {};
\path (4,0) node[draw,shape=circle] (c1) {};

\path (2,1.6) node[draw,shape=circle,fill=black] (P0) {};
\path (2,0.8) node[draw,shape=circle,fill=black] (P1) {};
\path (2,0) node[draw,shape=circle,fill=black] (P2) {};
\path (2,-2) node[draw,shape=circle,fill=black] (P3) {};

\draw[thick] (c1) -- (P0) -- (c0) -- (P1) -- (c1) -- (P2) -- (c0) -- (P3) -- (c1);
\draw (2,-1) node[rotate=90] {$\ldots$};
\end{tikzpicture}}}\\
\subfigure[$S_n$]{
\scalebox{1.0}{\begin{tikzpicture}[join=bevel,inner sep=0.5mm,scale=0.8,line width=0.5pt]
\path (0,0) node[draw,shape=circle,fill=black] (c0) {};
\path (2,0) node[draw,shape=circle] (c1) {};

\path (-1,3) node[draw,shape=circle,fill=black] (P0) {};
\path (0,3) node[draw,shape=circle,fill=black] (P1) {};
\path (1,3) node[draw,shape=circle,fill=black] (P2) {};
\path (3,3) node[draw,shape=circle,fill=black] (P3) {};

\draw[thick] (c0) -- (c1) -- (P0) -- (c0) -- (P1) -- (c1) -- (P2) -- (c0) -- (P3) -- (c1);
\draw (2,3) node {$\ldots$};
\end{tikzpicture}}}\qquad
\subfigure[$T_n$]{
\scalebox{1.0}{\begin{tikzpicture}[join=bevel,inner sep=0.5mm,scale=0.8,line width=0.5pt]
\path (-1.4,0) node[draw,shape=circle,fill=black] (c0) {};
\path (0,0) node[draw,shape=circle,fill=black] (c1) {};
\path (0,1.4) node[draw,shape=circle] (c2) {};

\path (-2,3) node[draw,shape=circle,fill=black] (P0) {};
\path (-1,3) node[draw,shape=circle,fill=black] (P1) {};
\path (0,3) node[draw,shape=circle,fill=black] (P2) {};
\path (2,3) node[draw,shape=circle,fill=black] (P3) {};

\draw[thick] (c0) -- (c1) -- (c2) -- (P1)
             (c2) -- (P2)
             (c2) -- (P3)
             (c2) -- (P0);
\draw (1,3) node {$\ldots$};
\end{tikzpicture}}}\qquad
\subfigure[$U_n$]{
\scalebox{1.0}{\begin{tikzpicture}[join=bevel,inner sep=0.5mm,scale=0.8,line width=0.5pt]
\path (-1,0) node[draw,shape=circle,fill=black] (c0) {};
\path (1,0) node[draw,shape=circle,fill=black] (c1) {};
\path (0,1.4) node[draw,shape=circle] (c2) {};

\path (-2,3) node[draw,shape=circle,fill=black] (P0) {};
\path (-1,3) node[draw,shape=circle,fill=black] (P1) {};
\path (0,3) node[draw,shape=circle,fill=black] (P2) {};
\path (2,3) node[draw,shape=circle,fill=black] (P3) {};

\draw[thick] (c2) -- (c0) -- (c1) -- (c2) -- (P1)
             (c2) -- (P2)
             (c2) -- (P3)
             (c2) -- (P0);
\draw (1,3) node {$\ldots$};
\end{tikzpicture}}}

\caption{The list of graphs on $n$ vertices with centroidal dimension $n-1$.}\label{fig:extremal}
\end{figure}
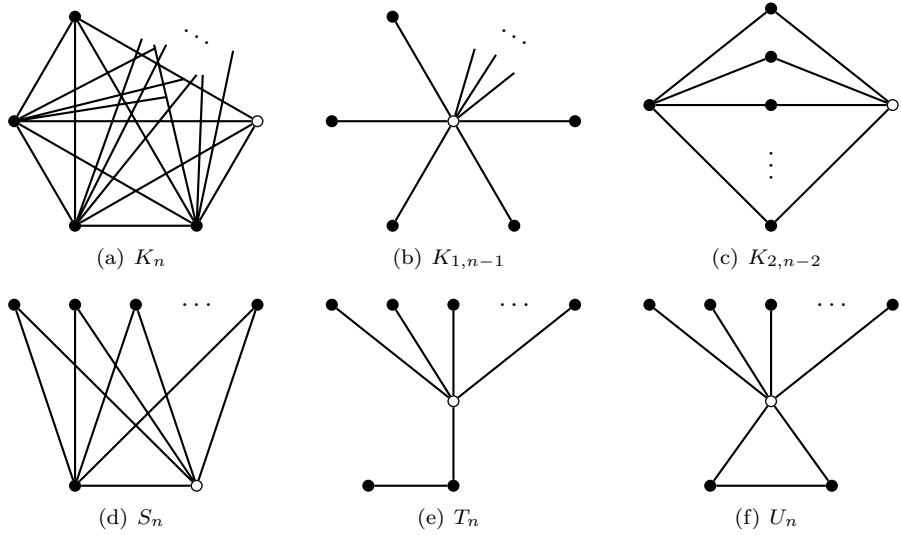

\begin{proposition}
Let $G$ be a graph on $n\geq 3$ vertices belonging to
$\{K_n,K_{1,n-1},K_{2,n-2},S_n,T_n,U_n\}$. Then $\M(G)=n-1$.
\end{proposition}
\begin{proof}
By Lemma~\ref{lemma:deg2} applied on a single degree~2-vertex, it is
easily seen that in each case, $\M(G)\leq n-1$. If $G$ is isomorphic
to $K_n$, the lower bound is directly implied by
Lemma~\ref{lemma:twins}(c).

If $G$ is isomorphic to $K_{1,n-1}$, then by Lemma~\ref{lemma:deg1}(a)
any centroidal locating set contains all $n-1$ leaves of $K_{1,n-1}$.

If $G$ is isomorphic to $K_{2,n-2}$ or to $S_n$, by
Lemma~\ref{lemma:twins}(c) at least $n-3$ vertices of degree~2 belong to
any centroidal locating set, as well as one of the other two
vertices. However, if no further vertex does belong to the centroidal
locating set, then the vertex of degree~2 that is not in the set is
not distinguished from its neighbour in the set, a contradiction.

If $G$ is isomorphic to $T_n$, by Lemma~\ref{lemma:deg1}(a) all $n-2$
vertices of degree~1 belong to any centroidal locating set. Moreover,
one further vertex also does according to Lemma~\ref{lemma:deg1-2}(b).

Finally, if $G$ is isomorphic to $U_n$, again by
Lemma~\ref{lemma:deg1}(a) all $n-3$ vertices of degree~1 belong to any
centroidal locating set, as well as one degree~2-vertex of the
triangle by Lemma~\ref{lemma:twins}(c). If no further vertex does belong
to a centroidal locating set, then the two degree~2-vertices of the
triangle are not distinguished, a contradiction.
\end{proof}

In fact, the graphs from Figure~\ref{fig:extremal} are the only
extremal graphs, as shown in the following theorem:

\begin{theorem}
Let $G$ be a connected graph on $n\geq 3$ vertices with
$\M(G)=n-1$. Then $G$ belongs to $\{K_n,K_{1,n-1},$
$K_{2,n-2},S_n,T_n,U_n\}$.
\end{theorem}
\begin{proof}
We assume by contradiction that $\M(G)=n-1$ but $G$ does not belong to
the list. Hence, $n\geq 5$ since all connected graphs on three or four
vertices are in the list. Let $u$ be a vertex of $G$ of degree at
least~2.  By Lemma~\ref{lemma:deg2}, $V(G)\setminus\{u\}$ is a
centroidal locating set. First of all, observe that there is no vertex
at distance~3 of $u$: for contradiction, assume that
$z$ is such a vertex and let $t$ be a neighbour of $z$ with
$deg(t)\geq 2$ that lies on a path from $u$ to $z$. By
Lemma~\ref{lemma:deg2}, $B=V(G)\setminus\{u,t\}$ is a centroidal
locating set, a contradiction.

\vspace{0.5cm} We now assume that $deg(u)=2$. Let $v,w$ be the two
neighbours of $u$, and let $x,y$ be two other vertices in $G$ (they
exist since $n\geq 5$). Since there is no vertex at distance~3 of $u$,
both $x,y$ are neighbours of at least one of $v,w$. First, assume that
$x$ is a neighbour of both $v$ and $w$ and that $y$ is only a
neighbour of $v$. Then $B=V(G)\setminus\{v,w\}$ is a centroidal
locating set, a contradiction. Indeed, each vertex of $B$ is first
located by itself, then by its neighbours in $B$, while $v,w$ are
first located by a set of at least two vertices. Moreover, $v$ is (in
particular) first located by $y$, while $w$ is not.

Hence, either all vertices other than $u,v,w$ are adjacent to both
$v,w$, or none is. In the first case, if all common neighbours of
$v,w$ form an independent set, $G$ is isomorphic to either $K_{2,n-2}$
or $S_n$, a contradiction. Hence, there is an edge between two common
neighbours of $v,w$, say between $x,y$. Then, $B=V(G)\setminus\{u,x\}$
is a centroidal locating set, a contradiction. Indeed, each vertex of
$B$ is first located by itself, then by its neighbours in $B$, while
$u,x$ are first located by a set of at least two vertices. Moreover,
$x$ is (in particular) first located by $y$, while $u$ is not.

We now have that each vertex other than $u,v,w$ is adjacent to exactly
one of $v,w$. Let $S_v$ be the set of neighbours of $v$ other than
$u,w$ and $S_w$ be the set of neighbours of $w$ other than $u,v$. If
both $|S_v|,|S_w|\geq 1$, by Lemma~\ref{lemma:deg2},
$V(G)\setminus\{v,w\}$ is a centroidal locating set, a
contradiction. Otherwise, either $G$ is isomorphic to $T_n$ if $v,w$
are non-adjacent, or to $U_n$ otherwise, a contradiction in both
cases.

\vspace{0.5cm} By the previous discussion, $G$ has no
degree~2-vertices. Hence $deg(u)\geq 3$. Suppose moreover that $N(u)$
is an independent set. Since $G$ is not isomorphic to $K_{1,n-1}$,
there are vertices at distance~2 of $u$. Hence $u$ has a neighbour,
$v$, with two such vertices as neighbours (since $deg(v)\neq 2$): let
$x,y$ be these vertices. Then, $B=V(G)\setminus\{u,v\}$ is a
centroidal locating set, a contradiction. Indeed, all vertices but
$u,v$ are located first by themselves only, whereas $u$ is located
first by $N(u)\setminus\{v\}$ and $v$ by a set containing both $x,y$.

Now, we assume that $N(u)$ is a clique. Since $G$ is not a complete
graph, $u$ has a neighbour $v$, having a neighbour $x$ with
$d(u,x)=2$. Then, by Lemma~\ref{lemma:deg2}, $V(G)\setminus\{u,v\}$ is
a centroidal locating set of $G$, a contradiction.

Hence, $N(u)$ is neither an independent set, nor a clique: there is a
vertex $v$ in $N(u)$ with a neighbour $w$ and a non-neighbour $x$,
both being in $N(u)$. Since no vertex in $G$ has degree~2, $v$ has an
additional neighbour, $y$. But then by Lemma~\ref{lemma:deg2},
$V(G)\setminus\{u,v\}$ is a centroidal locating set of $G$, a
contradiction.
\end{proof}

\section{Graphs with few paths connecting each pair of vertices}\label{sec:paths-cycles}

We now study parameter $\M$ for graphs in which every pair of vertices
is connected via a bounded number $k$ of paths. We show that such
graphs have centroidal dimension
$\Omega\left(\sqrt{\frac{m}{k}}\right)$, where $m$ is the number of
edges. In particular, this applies to paths and cycles; for these
graphs, we show that the lower bound is asymptotically tight. These
cases are particularly interesting for the following reason: the
metric dimension of a path or a cycle with $n$ vertices is easily seen
to be constant (1 for any path, 2 for any cycle), whereas the
location-domination number is linear (roughly $\frac{2}{5}$th of the
vertices~\cite{S88}). In contrast, the centroidal dimension is about
the square-root of the order.

But first, the following technical lemma will be useful when showing
our lower bound.

\begin{lemma}\label{lemma:LBcycle-path}
Let $G$ be a graph with $u,v$ two adjacent vertices, and let $B$ be a
centroidal locating set. Then, there are two vertices $b_1,b_2$ of $B$
and a path $P:b_1-u-v-b_2$ such that at least one of the following
properties hold:
\begin{enumerate}
\item $\{u,v\}=\{b_1,b_2\}$ and $P$ is the edge $\{u,v\}$;
\item $d(u,b_1)=d(u,b_2)$, $d(v,b_1)\neq d(v,b_2)$ (or, symmetrically,
  $d(v,b_1)=d(v,b_2)$, $d(u,b_1)\neq d(u,b_2)$ and $P$ contains a
  shortest path from $u$ to $b_1$ and a shortest path from $v$ to
  $b_2$;
\item $d(u,b_1)+1=d(u,b_2)$, $d(v,b_1)=d(v,b_2)+1$, $P$ contains a
  shortest path from $u$ to $b_1$ and a shortest path from $v$ to
  $b_2$, $P$ has odd length and $\{u,v\}$ is the middle edge of $P$.
\end{enumerate}
\end{lemma}
\begin{proof}
If both $u,v$ belong to $B$, we are in the first case and we are done.

Otherwise, since $B$ is a centroidal locating set, without loss of
generality there are two vertices $b,b'$ of $B$ such that $d(u,b)\leq
d(u,b')$ and $d(v,b)>d(v,b')$, or vice-versa.

\vspace{0.2cm}\noindent\textbf{Case a: \boldmath{$d(u,b)=d(u,b')$}.} We show that the second case of the statement holds.
If $v$ lies on a shortest path $P_u$ from $u$ to one of $b,b'$ (say,
$b$), we are done: then no shortest path from $u$ to $b'$ can go
through $v$ (otherwise $d(v,b)=d(v,b')$). Hence, setting $b=b_2$ and
$b'=b_1$, the concatenation of $P_u$ and any shortest path from $u$ to
$b'$ is a path $P$ satisfying the desired properties.

Hence, we assume that $v$ does not lie on a shortest path from $u$ to
$b$. Since $d(v,b)\neq d(v,b')$, we can assume without loss of
generality that $d(v,b')\neq d(u,b')+1$ (moreover since $d(v,b')\leq
d(u,b')+1$, we have $d(v,b')\leq d(u,b')$). Hence a shortest path
$P_{v}$ from $v$ to $b'$ does not go through $u$. Let $P_u$ be a
shortest path from $b$ to $u$. If the concatenation $P_u-uv-P_v$ is a
path, we are done by setting $P=P_u-uv-P_v$, $b=b_1$ and
$b'=b_2$. Therefore, assume that it is not a path, and let $w$ be the
vertex closest to $u$ appearing in both $P_u$ and $P_v$: $P_u-uv-P_v$
contains a cycle going through $u,v,w$. Now, since no shortest path
from $v$ to $b'$ goes through $u$, we have $d(v,w)\leq d(u,w)$. Also,
since we assumed that $v$ does not lie on a shortest path from $u$ to
$b$, in particular $d(u,w)\leq d(v,w)$. Therefore,
$d(u,w)=d(v,w)$. This implies that $d(v,b')=d(u,b')$ (indeed, we had
$d(v,b')\leq d(u,b')$ and now $d(u,b')\leq
d(u,w)+d(w,b')=d(v,w)+d(w,b')=d(v,b')$). Hence, $d(v,b)<d(v,b')$. Let
$P'_u$ be the path obtained from the concatenation of a shortest path
from $u$ to $w$ and the subpath of $P_v$ from $w$ to $b'$, and let
$P'_v$ be a shortest path from $v$ to $b$. Then, $P'_v$ does not
contain any vertex of $P'_u$ (indeed, if there was such a vertex $t$,
then $d(u,b)\leq d(v,b)<d(v,b')=d(u,b')$, a contradiction). Hence, the
concatenation $P=P'_v-uv-P'_u$ is a path that has the desired
properties.

\vspace{0.2cm}\noindent\textbf{Case b: \boldmath{$d(u,b)<d(u,b')$}.}
We show that the third case of the statement holds.  Since
$d(v,b)>d(v,b')$ and $u,v$ are adjacent, $d(v,b)=d(u,b)+1$ and
$d(v,b')=d(u,b')-1$. Hence $d(v,b')<d(v,b)<d(v,b')+2$ and
$d(u,b)<d(u,b')<d(u,b)+2$, and $d(u,b)=d(v,b')$. We set $b_1=b$ and
$b_2=b'$. Observe that the concatenation $P$ of a shortest path from
$b_1$ to $u$, the edge $uv$, and a shortest path from $v$ to $b_2$ is
a path from $b_1$ to $b_2$ (if it were not a path, we would have
$d(u,b)=d(u,b')=d(v,b)=d(v,b')$). Since $P$ has the desired
properties, this completes the proof.
\end{proof}

In what follows, for a pair $\{u,v\}$ of vertices in a graph $G$, we
let $k_{odd}(u,v)$ and $k_{ev}(u,v)$ be the number of odd and even
(not necessarily disjoint) paths connecting $u$ to $v$, respectively.

\begin{theorem}\label{thm:bounded-nr-paths}
Let $G$ be a graph on $n$ vertices and $m$ edges such that for every
pair $\{u,v\}$ of vertices, $2k_{ev}(u,v)+k_{odd}(u,v)\leq k$ for some
integer $k$. Then, $\M(G)>\sqrt{\frac{2m}{k}}$.  In particular, for every
  tree $T$, $\M(T)>\sqrt{n-1}$.
\end{theorem}
\begin{proof}
Let $B$ be a centroidal locating set of $G$. To each pair $u,v$ of
adjacent vertices in $G$, we assign a triple $(b_1,b_2,P)=T(u,v)$ of
two vertices $b_1,b_2$ of $B$ and a path $P:b_1-u-v-b_2$ satisfying
one of the three properties described in
Lemma~\ref{lemma:LBcycle-path}. The assignment is done as follows:
\begin{enumerate}
\item if both $u,v$  belong to $B$, set $T(u,v)=(u,v,uv)$ (we say that
  the pair $u,v$ is of \emph{type~1});
\item otherwise, if there is some pair $\{b_1,b_2\}$ of $B$ and the
  corresponding path $P$ such that the second property of
  Lemma~\ref{lemma:LBcycle-path} holds, then set $T(u,v)=(b_1,b_2,P)$
  (we say that the pair $u,v$ is of \emph{type~2});
\item otherwise, let $T(u,v)$ consist of an arbitrary pair $b_1,b_2$
  of vertices of $B$ and the corresponding path $P$ satisfying the
  third property of Lemma~\ref{lemma:LBcycle-path} (we say that the
  pair $u,v$ is of \emph{type~3}).
\end{enumerate}

Now, for a given pair $b_1,b_2$ of $B$ and a path $P$ from $b_1$ to
$b_2$, we will upper-bound the number of pairs of adjacent vertices
$u,v$ such that $T(u,v)=(b_1,b_2,P)$.

Assume first that $P$ has odd length. If $P$ is the edge $b_1,b_2$,
there is only the pair $\{u,v\}=\{b_1,b_2\}$ of type~1 with
$T(u,v)=(b_1,b_2,P)$. Otherwise, assume $\{u,v\}$ is a pair of type~2
with $T(u,v)=(b_1,b_2,P)$. Then by Lemma~\ref{lemma:LBcycle-path},
$d(u,b_1)=d(u,b_2)=\ell_1$ and the subpath of $P$ from $u$ to $b_1$
has length $\ell_1$. Since the subpath from $v$ to $b_2$ is also a
shortest path and $d(u,b_2)=\ell_1$, it has length $\ell_2$ with
$\ell_1-1\leq\ell_2\leq\ell_1+1$. The length of $P$ is
$\ell_1+\ell_2+1$. Since it is an odd number, $\ell_2=\ell_1$ and
$\{u,v\}$ is the middle edge of $P$. If $\{u,v\}$ is of type~3, by
Lemma~\ref{lemma:LBcycle-path}, $\{u,v\}$ is also the middle edge of
$P$. Hence in total there is at most one pair $u,v$ with
$T(u,v)=(b_1,b_2,P)$.

Now, assume that $P$ has even length. Let $\{u,v\}$ be a pair with
$T(u,v)=(b_1,b_2,P)$. Then, by Lemma~\ref{lemma:LBcycle-path}, $u,v$
cannot be of type~1 or type~3, so it must be of type~2. By the same
arguments as in the previous paragraph, the length of the subpath of
$P$ from $b_1$ to $u$ is $\ell_1$, and the length of the subpath of
$P$ from $v$ to $b_2$ is either $\ell_1-1$ or $\ell_1+1$. In both
cases, one of $u,v$ is the middle vertex of $P$. Hence there can be at
most two such pairs.

To summarize, we proved that for each pair $b_1,b_2$ of $B$, the
number of pairs $u,v$ with $T(u,v)=(b_1,b_2,P)$ for some path $P$ is
at most $$2k_{ev}(b_1,b_2)+k_{odd}(b_1,b_2)\leq k.$$

Since in total, there are $m$ pairs of adjacent vertices in $G$ and
each such pair is associated to exactly one of the $\binom{|B|}{2}$
pairs of $B$, we have $m\leq k\binom{|B|}{2}<\frac{k|B|^2}{2}$, and
the bound follows.

When $G$ is a tree, there are $n-1$ edges, and there is a unique path
between any pair of vertices, hence $k\leq 2$.
\end{proof}

We now show that the bound of Theorem~\ref{thm:bounded-nr-paths} is
tight up to a constant factor for paths and cycles.

\begin{theorem}\label{thm:paths-cycles}
Let $n\geq 3$.\\
(1) If $n$ is even,
$\frac{\sqrt{2}}{2}\sqrt{n}<\M(C_n)$. If $n$ is odd,
$\frac{\sqrt{6}}{3}\sqrt{n}<\M(C_n)$. In both cases,
$\M(C_n)<\frac{7\sqrt{2n}}{2}+1$. If $n=2\ell^2$ for some
integer $\ell$, then $\M(C_n)\leq\sqrt{2n}-2$.\\
(2) $\sqrt{n-1}<\M(P_n)<6\sqrt{n-1}+3$. If $n=(2\ell)^2+1$ for some
integer $\ell$, then $\M(P_n)\leq 2\sqrt{n-1}-3$.
\end{theorem}
\begin{proof}\textbf{Lower bounds.} The bounds follow from Theorem~\ref{thm:bounded-nr-paths} since paths are trees, and by observing that cycles have $n$ edges. For even cycles, for any pair $\{u,v\}$ of vertices, $k_{odd}(u,v)\leq 2$ and $k_{ev}(u,v)\leq 2$ and $k_{odd}(u,v)+k_{ev}(u,v)=2$, hence $2k_{ev}(u,v)+k_{odd}(u,v)\leq 4$. For odd cycles, $k_{odd}(u,v)=1$ and $k_{ev}(u,v)=1$, hence $2k_{ev}(u,v)+k_{odd}(u,v)\leq 3$.

\vspace{0.5cm}\noindent \textbf{Upper bounds for cycles.}  We first prove that
for any $p,q\geq 2$, if $n=p(2q+2)$, then:

\begin{align}
\M(C_n)\leq p+q-1.\label{eq:Cn}
\end{align}

Assuming that $n=2\ell^2$, and setting $p=\ell$ and $q=\ell-1$,
Inequality~\eqref{eq:Cn} yields the claimed bound $\M(C_n)\leq
2\ell-2=\sqrt{2n}-2$.

Let $\{x_0,\ldots,x_{n-1}\}$ be the vertex set of $C_n$. Let us divide
$C_n$ into $p$ portions of $2q+2$ consecutive vertices each: for
$0\leq i\leq p-1$,
$R_i=\{x_{i(2q+2)},x_{i(2q+2)+1},\ldots,x_{(i+1)(2q+2)-1}\}$. For each
$i$, we further define two subsets of $R_i$ as follows:
$S_i=\{s_i^1,\ldots,s_i^{q}\}$ and $T_i=\{t_i^1,\ldots,t_i^{q}\}$,
where for $1\leq j\leq q$, $s_i^j=x_{i(2q+2)+j}$ and
$t_i^j=x_{i(2q+2)+q+1+j}$. In other words, $S_i$ contains $q$ vertices
from $R_i$, starting from the second one, and $T_i$ contains the $q$
last vertices of $R_i$. Observe that the first and the $(q+2)$-nd
vertices from $R_i$ neither belong to $S_i$, nor $T_i$.

We now define a set $B=B_0\cup B_1$, which we claim, will be our
centroidal locating set. We let $B_0=\{b_0^0,\ldots,b_0^{q-1}\}$,
where for $0\leq i\leq q-1$, $b_0^i=x_{2(i+1)}$ (that is, $B_0$ contains
each second vertex of $R_0$). We let $B_1=\{b_1^0,\ldots,b_1^{p-1}\}$,
where for $0\leq i\leq p-1$, $b_1^i=x_{i(2q+2)}$ (that is, $B_1$ contains
the first vertex of each set $R_i$). 

An illustration of sets $S_i,T_i,B_0,B_1$ is given in
Figure~\ref{fig:cycleUB}.

\begin{figure}[htpb!]
\centering
\includegraphics[width=0.9\textwidth]{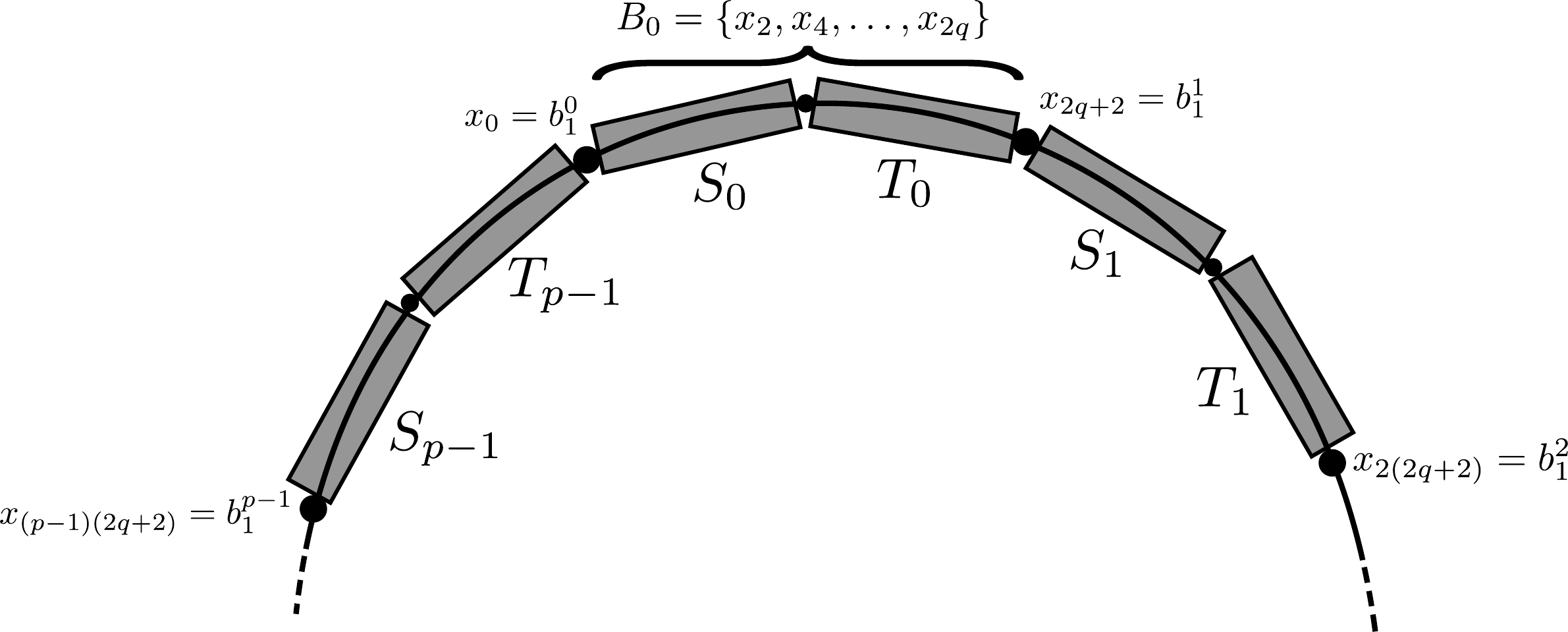}
\caption{Illustration of the sets defined in the proof of Theorem~\ref{thm:paths-cycles}.}
\label{fig:cycleUB}
\end{figure}

Observe that $B$ has $p+q-1$ elements. It remains to show that $B$ is
a centroidal locating set.

First of all, notice that for any $0\leq i\leq p-1$, $b_1^i\in B_1$ is
the unique vertex that is first located by itself, and later by
$\{b_1^{(i-1)\bmod p},b_1^{(i+1)\bmod p}\}$ at the same time. For
$i\neq 0$, the $(q+2)$-nd vertex of $R_i$, $x_{i(2q+2)+q+1}$, is the
only one that is located \emph{first} by $\{b_1^{i},b_1^{(i+1)\bmod
  p}\}$ at the same time. Similarly, $x_{q+1}$ is the unique vertex
first located by the vertices of $B_0$ (in some order), and then by
$\{b_1^0,b_1^1\}$ at the same time. If $x\in S_0\cup T_0$, if $x\in
B_0$, $x$ is the only vertex located first by itself only; if $x\notin
B_0$, $x$ is the only vertex located first by its two
neighbours. Hence, all the previously considered vertices are
distinguished from any other vertex in $C_n$.

Now, let $u,v$ be a pair of vertices not yet proved to be
distinguished.  If $u\in S_i, v\in T_i$ ($1\leq i\leq p-1$), then
$d(u,b_1^i)<d(u,b_1^{(i+1)\bmod p})$ but
$d(v,b_1^i)>d(v,b_1^{(i+1)\bmod p})$. If $u\in S_i, v\in S_{i'}$
($1\leq i,i'\leq p-1$), then $d(u,b_1^i)<d(u,b_1^{i'})$ but
$d(v,b_1^i)>d(v,b_1^{i'})$. The case where $u\in T_i,v\in T_{i'}$ is
symmetric. If $u\in S_i, v\in T_{i'}$ ($1\leq i,i'\leq p-1$), if
$i=(i'+1)\bmod p$, then $d(u,b_1^{(i+1)\bmod p})<d(u,b_1^{i'})$ while
$d(v,b_1^{(i+1)\bmod p})>d(v,b_1^{i'})$. Otherwise,
$d(u,b_1^{i})<d(u,b_1^{(i'+1)\bmod p})$ while
$d(v,b_1^{i})>d(v,b_1^{(i'+1)\bmod p})$.

It remains to prove that for $1\leq i\leq p-1$, any two vertices from
$S_i$ are distinguished (the case of two vertices of $T_i$ would
follow by symmetry). To see this, let $s_i^j\in S_i$. If
$i\leq\left\lceil\frac{p-1}{2}\right\rceil$, $d(s_i^j,
b_1^{2i})=d(s_i^j,b_0^j)=(2q+2)i-j$, and no other vertex of $S_i$ has
this property. Similarly, if $i>\left\lceil\frac{p-1}{2}\right\rceil$,
$d(s_i^j, b_1^{2i-p})=d(s_i^j,b_0^j)=(2q+2)(p-i)+j$. This completes
the proof of validity of $B$.

\vspace{0.5cm} In order to prove the bound for all cycles, if $n$ is
not of the form $2\ell^2$, let $m=2\ell^2$ be the integer of this form
that is closest to $n$ and such that $m\leq n$: $n=m+k$ for some $k$,
and $n<2(\ell+1)^2$.  A construction similar to the previous one can
be done. Letting $p=\ell$, $q=\ell-1$, we construct
$B_1=\{b_1^0,\ldots,b_1^{p-1}\}$ with $b_1^i=x_{i(2q+2)}$ as
previously; however this time $B_1$ does not include any vertex from
$\{x_m,\ldots,x_{n-1}\}$. Instead, we let
$B_0=\{x_1,\ldots,x_{2q+1}\}\cup\{x_m,\ldots,x_{n-1}\}$: $B_0$
contains the first $2q+2$ vertices (except $x_0$), together with the
$k$ last vertices. It is clear that the construction works the same
way than in the previous proof --- we omit the details as a formal
proof would be tedious.\footnote{In fact we have not taken care of
  optimizing the construction, as here all the vertices of $B_0$ are
  probably not needed.} In total, $B=B_0\cup B_1$ has size at most
$p+2q+1+k=3\ell-1+k$. Since $2(\ell+1)^2-2\ell^2=4\ell+2$, $k<4\ell+2$
and $|B|\leq 7\ell+1$. Furthermore $2\ell^2<n$, and
$\ell<\frac{\sqrt{2n}}{2}$. Hence $|B|<\frac{7\sqrt{2n}}{2}+1$,
proving the bound.

\vspace{0.5cm}\noindent\textbf{Upper bounds for paths.}  A similar
construction than the one for cycles can be done for the case of
paths: for any $p,q\geq 2$, if $n=p(2q+2)+1$, then $\M(P_n)\leq
p+2q-1$. The idea of the construction is again to divide the vertex
set $\{x_0,\ldots,x_{n-1}\}$ into $p$ portions of size $2q+2$ each
(vertex $x_{n-1}$ does not belong to any such portion). For $0\leq
i\leq p$, $b_1^i=x_{i(2q+2)}$, and $B_1$ contains all vertices of the
form $b_1^i$. $B_0$ is defined in the same way as for our construction
for cycles, but now we also consider a set $B'_0$ similar to $B_0$,
but on the other end of the path. We have $|B_0\cup B'_0\cup
B_1|=p+1+2(q-1)=p+2q-1$, and similar arguments than for the
construction for cycles show that $B_0\cup B'_0\cup B_1$ is a
centroidal locating set.

Hence, assuming $n=(2\ell)^2+1$ for some $\ell\geq 2$ and setting
$p=2\ell$ and $q=\ell-1$, we get that $\M(P_n)\leq
4\ell-3=2\sqrt{n-1}-3$.

For the general bound, once again we do not optimize the
constant. Assume that $n$ is not of the form $(2\ell)^2+1$, and
let $m=(2\ell)^2+1$ be the integer of this form that is closest to $n$
and $m\leq n$: we have $n=m+k$ for some $k$, and
$n<(2(\ell+1))^2+1$. Let $p=2\ell$ and $q=\ell-1$. Now, $B_1$ is
selected as before among the first $m$ vertices; $B_1$ has $p$
elements. $B_0$ contains the first $2q+2$ vertices (except $x_0$), and
$B'_0$ contains the last $k+2q+3$ vertices (except $x_m$). In total,
$B=B_0\cup B'_0\cup B_1$ has $p+(2q+1)+(k+2q+2)=p+4q+3+k=6\ell-1+k$
vertices. Since $(2(\ell+1))^2+1-(2\ell)^2-1=6\ell+4$,
$k<6\ell+4$. This implies $|B|<12\ell+3$. Since $n>(2\ell)^2+1$,
$\ell<\frac{\sqrt{n-1}}{2}$; hence, $|B|<6\sqrt{n-1}+3$.
\end{proof}

\section{Complexity results}\label{sec:complex}

Let us now turn our attention to the computational complexity of
finding a small centroidal locating set, that is, the computational
complexity of the following problem:\\

\optpb{\textsc{Centroidal Dimension}}{A graph $G$.}{Find a
  centroidal basis of $G$.}

We have seen in Theorem~\ref{thm:bound-diam2} that for any graph of
diameter~2, $\LD(G)-1\leq\M(G)\leq 2\LD(G)$. We get the following
corollary, showing that \textsc{Centroidal Dimension} is
computationally very hard, even from the approximation point of view
(recall that an $\alpha$-approximation algorithm for problem $P$ is a
polynomial-time algorithm for $P$ which always outputs a solution of
size no greater than $\alpha$ times the size of an optimal solution).

\begin{corollary}\label{cor:complexity}
\textsc{Centroidal Dimension} is \textsf{NP}-hard to
approximate within any factor $o(\ln n)$ for graphs on $n$ vertices
(even for graphs with a vertex adjacent to all other
vertices, and hence diameter~2-graphs). For graphs of
diameter~2, it has an $O(\ln n)$-approximation algorithm.
\end{corollary}
\begin{proof}
Since $\LD(G)-1\leq\M(G)\leq 2\LD(G)$ and the bounds are constructive,
any $\alpha$-approximation algorithm ($\alpha\geq 1$) for
\textsc{Minimum Locating-Dominating Set} can be transformed into an
approximation algorithm of factor
$2\alpha(1+\frac{1}{OPT})=2\alpha(1+o(1))$ for \textsc{Centroidal
  Dimension} for graphs of diameter~2, and vice-versa. Indeed, given an
$\alpha$-approximate locating-dominating set $D$ of $G$, we construct
a centroidal locating set of $G$ of size at most $2|D|$. We have
$2|D|\leq 2\alpha\LD(G)\leq 2\alpha
(\M(G)+1)=(2\alpha+\frac{2\alpha}{\M(G)})\M(G)$. For the converse, the
reasoning is similar.

This also implies that if \textsc{Minimum Locating-Dominating Set} is
\textsf{NP}-hard to $\alpha$-approximate for graphs of diameter~2 for
some $\alpha\geq 2$, then \textsc{Centroidal Dimension} is
\textsf{NP}-hard to approximate within factor
$\left(\frac{\alpha}{2}\frac{OPT}{OPT+1}\right)=\frac{\alpha(1-o(1))}{2}$
for graphs of diameter~2.

The positive approximation bound follows, as \textsc{Minimum
  Locating-Dominating Set} is well-known to be
$O(\ln n)$-approximable, see for example Gravier, Klasing and
Moncel~\cite{GKM08}.

Moreover, it follows from a reduction for \textsc{Minimum Identifying
  Code} in the first author's thesis~\cite[Section 6.4]{F} and a lemma
from Gravier, Klasing and Moncel~\cite{GKM08} (see also
Foucaud~\cite{F13a,F13}) that \textsc{Minimum Locating-Dominating Set}
is \textsf{NP}-hard to approximate within a factor of $o(\ln n)$ for
graphs having a vertex adjacent to all other vertices. This proves the
non-approximability bound.
\end{proof}

Note that Corollary~\ref{cor:complexity} fully determines the
computational complexity of \textsc{Centroidal Dimension} in
graphs of diameter~2 from the approximation point of view. It was
recently proved by Hartung and Nichterlein that the related problem
\textsc{Metric Dimension} remains \textsf{NP}-hard to
approximate within a factor of $o(\ln n)$ even for subcubic
graphs~\cite{HN}. In general, it would be interesting to extend the
result of Corollary~\ref{cor:complexity} to other families of graphs.

As it is often the case with domination or identification problems in
graphs, a good way of reformulating our problem is to represent it as
an instance of \textsc{Minimum Set Cover}, which is well-known to be
$(\ln n+1)$-approximable (see Johnson~\cite{J74}):\\

\optpb{\textsc{Minimum Set Cover}}{A hypergraph $H=(V,E)$.}{Find a minimum-size subset $C\subseteq E$
  such that $\bigcup_{X\in C}X=V$.}

For example, \textsc{Metric Dimension} for a graph $G$ can be
expressed in this way by constructing a hypergraph $H_{MD}(G)$ on
vertex set $\binom{V(G)}{2}$ with a hyperedge $E_v$ for each vertex
$v\in V(G)$ containing all pairs $\{x,y\}$ of vertices with
$d(x,v)\neq d(y,v)$. Then $H_{MD}(G)$ has a set cover of size $k$ if
and only if $G$ has a locating set of size $k$, as shown by Khuller,
Raghavachari and Rosenfeld~\cite{KRR96}. Hence \textsc{Metric
  Dimension} is $O(\ln n)$-approximable.

Next, we give a similar reduction for \textsc{Centroidal Dimension},
but with a weaker approximation ratio.

\begin{theorem}\label{thm:sqrtn-approx}
\textsc{Centroidal Dimension} is
$O\left(\sqrt{n\ln n}\right)$-approximable for graphs on $n$
vertices.
\end{theorem}
\begin{proof}
Using Observation~\ref{obs:otherdef}, finding a centroidal locating
set is equivalent to finding a set of \emph{pairs} of vertices
which identifies each pair of vertices in $G$ --- where a pair $b_1,b_2$
identifies $x,y$ if $d(x,b_1)\leq d(x,b_2)$ but $d(y,b_1)>d(y,b_2)$,
or $d(y,b_1)\leq d(y,b_2)$ and $d(x,b_1)>d(x,b_2)$.

Let $G$ be a graph on $n$ vertices. We define the hypergraph
$H=H_{CD}(G)$ on vertex set $V(H)=\binom{V(G)}{2}$. For each pair
$\{x,y\}$ of vertices of $G$, $H$ has a hyperedge $E_{x,y}$ that
contains all pairs that are identified by $\{x,y\}$.

Let $C$ be a set cover of $H$. Using our previous observation, one can
construct a centroidal locating set $B(C)$ of $G$ by taking the union
of all elements in the pairs that correspond to hyperedges in $C$:
$B(C)=\{x\in V(G)~|~\exists y\in V(G), E_{x,y}\in C\}$. Indeed, every
pair of vertices of $G$ is identified by a pair corresponding to a
hyperedge of $H$. Hence we have:
\begin{align}
|B(C)|\leq 2|C|.\label{eq:approx1}
\end{align}

For the other direction, given a centroidal basis $B$ of $G$, one can
construct a set cover of $H$ consisting of all $\binom{|B|}{2}$ pairs
of vertices of $B$: each pair $u,v$ in $V(G)$ is identified by some
pair $x,y$ in $B$, hence vertex $\{u,v\}$ in $H$ is covered by the
corresponding hyperedge $E_{x,y}$. Denoting by $\SC(H)$ the size of an
optimal set cover of $H$, this implies:
\begin{align}
\SC(H)\leq\binom{|B|}{2}\leq\binom{\M(G)}{2}.\label{eq:approx2}
\end{align}

Now, in order to approximate \textsc{Centroidal Dimension},
we construct $H=H_{CD}(G)$ from $G$, and apply the standard approximation
algorithm for \textsc{Minimum Set Cover}~\cite{J74} to get a set cover $C$ of $H$
of size $O\left(\ln n\cdot\SC(H)\right)$. By
Inequalities~\eqref{eq:approx1} and~\eqref{eq:approx2}, we get a
centroidal locating set $B(C)$ of size at most
$2|C|=O\left(\ln n\cdot\SC(H)\right)=O\left(\ln n\cdot\M(G)^2\right)$.

Now, if $\M(G)\geq\sqrt{\frac{n}{\ln n}}$, we have a trivial
$O\left(\sqrt{n\ln n}\right)$-approximation by selecting all vertices
as a solution, since $n= O\left(\sqrt{n\ln
  n}\cdot\M(G)\right)$.

If $\M(G)\leq\sqrt{\frac{n}{\ln n}}$,
$|B(C)|=O\left(\ln n\cdot\M(G)\cdot\sqrt{\frac{n}{\ln n}}\right)=O\left(\sqrt{n\ln n}\cdot\M(G)\right)$.
\end{proof}

We note that the quadratic dependence between $\SC(H)$ and $\M(G)$ is
necessary for Inequality~\eqref{eq:approx2} in the reduction of
Theorem~\ref{thm:sqrtn-approx}. Indeed, when we considered the case of
cycles in Section~\ref{sec:paths-cycles}, we had $n$ special pairs to
distinguish using other pairs $b_1,b_2$ from $B$, but any pair
$b_1,b_2$ could only distinguish a small (constant) number of these
$n$ pairs. However, we could build a centroidal locating set $B$ of
size $O\left(\sqrt{n}\right)$, meaning that a large fraction of the
pairs from $B$ were indeed necessary to distinguish the pairs. Hence
this would lead to a set cover of $H$ of size $\Omega(|B|^2)$ in our
reduction. This suggests that one cannot improve the approximation
ratio from Theorem~\ref{thm:sqrtn-approx} for \textsc{Centroidal
  Dimension} by using our reduction. Hence we ask the following
question:

\begin{question}\label{qu:approx}
What is the exact approximation complexity of \textsc{Centroidal Dimension}?
\end{question}

We close the section with a remark on the parameterized complexity of
\textsc{Centroidal Dimension}: this problem is
fixed-parameter-tractable with parameter $k$, the size of the
solution, that is, it admits an algorithm of running time
$f(k)n^{O(1)}$ for some computable function~$f$:

\begin{proposition}\label{prop:FPT}
\textsc{Centroidal Dimension} is fixed-parameter-tractable
when parameterized by the size of the solution.
\end{proposition}
\begin{proof}
We know that the order $n$ of a graph with a centroidal locating set
of size $k$ is at most $b(k)$, hence, if the input has more thatn
$b(k)$ vertices, we answer NO. Otherwise, we enumerate all
$\binom{n}{k}$ subsets of vertices of size~$k$ to check if one of them
is a centroidal locating set (checking whether a given set is
centroidal locating can be done in time $n^{O(1)}$). This algorithm
has running time $\binom{n}{k}n^{O(1)}\leq n^{k+O(1)}=b(k)^{k+O(1)}$,
which is computable and only depends on $k$.
\end{proof}

In contrast to Proposition~\ref{prop:FPT}, it was proved by Hartung
and Nichterlein~\cite{HN} that deciding whether there is a solution of
size $k$ for \textsc{Metric Dimension} is highly unlikely to
be solvable by an algorithm of running time $n^{o(k)}$.


\vspace{0.5cm}
\noindent\textbf{Acknowledgements.} We thank the anonymous referees
for carefully reading the paper, especially for detecting a flaw in
the original proof of Theorem~\ref{thm:bounded-nr-paths} and in an
erroneous construction that was replaced by Theorems~\ref{thm:diam3}
and~\ref{thm:diam3-construct} (which were suggested by one of the
referees).


\begin{thebibliography}{11}

\bibitem{BC} R. F. Bailey and P. J. Cameron, Base size, metric dimension and other invariants of groups and graphs, Bull Lond Math Soc 43 (2011), 209--242.

\bibitem{Carson} D. I. Carson, On generalized location domination, Graph Theory, Combinatorics and Applications: Proc 7th Quadrennial International Conference on the Theory and Applications of Graphs, Vol. 1, 1995, pp. 161--179.

\bibitem{CEJO00} G.~Chartrand, L.~Eroh, M.~Johnson, and O.~Oellermann, Resolvability in graphs and the metric dimension of a graph, Discrete Appl Math 105(1-3) (2000), 99--113.

\bibitem{DGMP11} O. Delmas, S. Gravier, M. Montassier, and A. Parreau, On two variations of identifying codes, Discrete Math 311(17) (2011), 1948--1956.


\bibitem{F} F.~Foucaud, Combinatorial and algorithmic aspects of identifying codes in graphs, PhD thesis, Universit\'{e} Bordeaux 1, Bordeaux, France, 2012. Available online at \url{http://tel.archives-ouvertes.fr/tel-00766138}.

\bibitem{F13a} F.~Foucaud, The complexity of the identifying code problem in restricted graph classes, Proc 24th International Workshop on Combinatorial Algorithms, Lecture Notes in Computer Science, Vol. 8288, Springer, 2013, pp. 150--163.

\bibitem{F13} F.~Foucaud, On the decision and approximation complexities for identifying codes and locating-dominating sets in restricted graph classes, Manuscript (2013). Available online at \url{http://hal.archives-ouvertes.fr/hal-00923356}.



\bibitem{GKM08} S. Gravier, R. Klasing, and J. Moncel, Hardness results and approximation algorithms for identifying codes and locating-dominating codes in graphs, Algorithmic Oper Res 3(1) (2008), 43--50.

\bibitem{H64} S.~L.~Hakimi, Optimal distribution of switching centers and the absolute centers and medians of a graph, Oper Res 12 (1964), 462--475.

\bibitem{H59} F. Harary,  Status and contrastatus, Sociometry 22 (1959), 23--43.


\bibitem{HM76} F.~Harary and R.~A.~Melter, On the metric dimension of a graph, Ars Combin 2 (1976), 191--195.

\bibitem{HN} S. Hartung and A. Nichterlein, On the parameterized and  approximation hardness of metric dimension, Proc 2013 IEEE Conference on Computational Complexity, 2013, pp. 266--276.



\bibitem{HHH06} T.~W.~Haynes, M.~A.~Henning, and J.~Howard, Locating and total dominating sets in trees, Discrete Appl Math 154(8) (2006), 1293--1300.


\bibitem{HMMPSW10} M. C. Hernando, M. Mora, I. M. Pelayo, C. Seara, and D. R. Wood, Extremal graph theory for metric dimension and diameter, Electr J Comb 17(1) (2010), \#R30.


\bibitem{HLR02} I. Honkala, T. Laihonen, and S. Ranto, On strongly identifying codes, Discrete Math 254 (2002), 191--205.

\bibitem{J74} D.~S.~Johnson, Approximation algorithms for combinatorial problems, Journal Comput System Sci 9 (1974), 256--278.

\bibitem{J69} C. Jordan, Sur les assemblages de lignes (in French), J Reine Angew Math 70(2) (1869), 185--190.


\bibitem{KCL98} M.~G.~Karpovsky, K.~Chakrabarty, and L.~B.~Levitin, On a new class of codes for identifying vertices in graphs, IEEE Trans Inform Theor 44 (1998), 599--611.

\bibitem{KRR96} S. Khuller, B. Raghavachari, and A. Rosenfeld, Landmarks in graphs, Discrete Appl Math 70 (1996), 217--229.


\bibitem{biblio} A.~Lobstein, Watching systems, identifying, locating-dominating and discriminating codes in graphs: a bibliography, webpage. \url{http://www.infres.enst.fr/~lobstein/debutBIBidetlocdom.pdf}

\bibitem{sloane} OEIS Foundation Inc, The On-Line Encyclopedia of Integer Sequences, webpage. \url{http://oeis.org/A000670}.

\bibitem{ST04} A. Seb\H{o} and E. Tannier, On metric generators of graphs, Math Oper Res 29(2) (2004), 383--393.

\bibitem{SS10} S.~J.~Seo and P.~J.~Slater, Open neighborhood locating-dominating sets, Australasian J Combin 46 (2010), 109--120.

\bibitem{SS11} S.~J.~Seo and P.~J.~Slater, Open neighborhood locating–dominating in trees, Discrete Appl Math 159 (2011), 484--489.

\bibitem{S75} P.~J.~Slater, Leaves of trees, Congr Numer 14 (1975), 549--559.

\bibitem{S75b} P. J. Slater, Maximin facility location, J Res Nat Bur Stand Ser B 79 (1975), 107--115.

\bibitem{S80} P. J. Slater, Medians of arbitrary graphs, J Graph Theory 4 (1980), 389--392.


\bibitem{S87} P.~J.~Slater, Domination and location in acyclic graphs, Networks 17(1) (1987), 55--64.

\bibitem{S88} P.~J.~Slater, Dominating and reference sets in a graph, J Math Phys Sci 22(4) (1988), 445--455.


\bibitem{S99} P. J. Slater, A survey of sequences of central subgraphs, Networks 34(4) (1999), 244--249.

\bibitem{SS99} C. Smart and P. J. Slater. Center, median, and centroid subgraphs, Networks 34(4) (1999), 303--311.

\bibitem{orderedBell} H. S. Wilf, ``generatingfunctionology'', A. K. Peters Ltd, Wellesley, MA, third edition, 2006.

\bibitem{Z68} B. Zelinka, Medians and peripherians of trees, Arch Math 4(2) (1968), 87--95.

\end{thebibliography}
\end{document}